\documentclass[12pt,reqno]{amsart}
\usepackage[headings]{fullpage}
\usepackage{amsmath,amssymb,amsthm}
\usepackage{bbold}
\usepackage[bookmarks=true,%
colorlinks=true,%
linkcolor=blue,%
citecolor=blue,%
filecolor=blue,%
menucolor=blue,%
urlcolor=blue,%
breaklinks=true]{hyperref}
\usepackage{graphicx,subcaption,url}
\usepackage{fullpage,comment}
\usepackage[shortlabels]{enumitem}
\setlist{font=\normalfont}
\usepackage{mathrsfs,mathtools,scalerel}
\usepackage{tikz,tikz-cd}
\usetikzlibrary{knots,hobby,math,decorations.markings}
\usepackage{ifthen}
\allowdisplaybreaks

\makeatletter
\renewcommand*{\fps@figure}{htpb!}
\makeatother

\newlist{enumcc}{enumerate}{1}
\setlist[enumcc]{leftmargin=0pt, itemindent=*}
\setlist{font=\normalfont}

\newtheorem{theorem}{Theorem}[section]
\theoremstyle{definition}
\newtheorem{proposition}[theorem]{Proposition}
\newtheorem{lemma}[theorem]{Lemma}
\newtheorem{definition}[theorem]{Definition}
\newtheorem{remark}[theorem]{Remark}
\newtheorem{corollary}[theorem]{Corollary}

\DeclareMathAlphabet{\AMSbb}{U}{msb}{m}{n}

\tikzset{knot diagram/every knot diagram/.style={
    background color=gray!20,clip width=5,end tolerance=5pt,clip radius=0.1cm}}

\tikzset{-o-/.code 2 args={
    \pgfkeysalso{decoration={markings,mark=at position #1 with {\arrow{#2}}},
      postaction={decorate}}
}}
\newcommand{\tikzbase}[1]{\node (ref) at (#1){\phantom{$-$}}}

\providecommand{\abs}[1]{\left|{#1}\right|}

\DeclareMathOperator{\tr}{tr}
\DeclareMathOperator{\diag}{diag}
\DeclareMathOperator{\id}{id}
\DeclareMathOperator{\Span}{span}

\DeclareMathOperator{\Hom}{Hom}
\DeclareMathOperator{\End}{End}
\DeclareMathOperator{\Aut}{Aut}

\DeclareMathOperator{\Ad}{Ad}

\newcommand{\ints}{\AMSbb{Z}}

\newcommand{\cx}{\AMSbb{C}}

\newcommand{\surface}{\Sigma}
\newcommand{\surclose}{\widehat{\surface}}
\newcommand{\punctures}{\mathcal{P}}

\newcommand{\triang}{\Delta}

\newcommand{\iunit}{\mathsf{i}}

\newcommand{\onevec}{\mathbf{1}}

\newcommand{\ideal}[1]{\langle#1\rangle}
\newcommand{\qtorus}{\AMSbb{T}}
\newcommand{\rdtorus}{\overline{\qtorus}\vphantom{\qtorus}}
\newcommand{\CF}{\mathrm{CF}}
\newcommand{\CFr}{\mathrm{CFr}}
\newcommand{\Kas}{\mathrm{K}}
\newcommand{\FGA}{\mathcal{A}}
\newcommand{\FGX}{\mathcal{X}}

\newcommand{\PSL}{\mathrm{PSL}}
\newcommand{\SL}{\mathrm{SL}}
\newcommand{\GL}{\mathrm{GL}}
\newcommand{\Sp}{\mathrm{Sp}}
\newcommand{\gl}{\mathfrak{gl}}
\newcommand{\embed}{\hookrightarrow}

\def\BZ{\AMSbb Z}

\def\BC{\AMSbb C}

\def\be{\begin{equation}}
\def\ee{\end{equation}}

\def\BLWY{Bonahon--Liu--Wong--Yang\,}

\newcommand{\facedef}[2]{
\begin{tikzpicture}[scale=0.8,baseline=0.28cm]
\draw[fill=gray!20!white] (0,1)--(0.6,0)--(1,1);
\draw[inner sep=0pt] (0.1,0.5)node{\vphantom{$b$}#1} (1,0.5)node{\vphantom{$b$}#2};
\draw[fill=white] (0.6,0)circle(2pt);
\end{tikzpicture}
}

\begin{document}
\title[A relation between the BB and the BLWY invariants]{
  A relation between the Baseilhac--Benedetti and the Bonahon--Liu--Wong--Yang
  invariants}

\author{Stavros Garoufalidis}
\address{
Shenzhen International Center for Mathematics, Department of Mathematics \\
Southern University of Science and Technology \\
1088 Xueyuan Avenue, Shenzhen, Guangdong, China \newline
{\tt \url{http://people.mpim-bonn.mpg.de/stavros}}}
\email{stavros@mpim-bonn.mpg.de}

\author{Tao Yu}
\address{Shenzhen International Center for Mathematics \\
Southern University of Science and Technology \\
1088 Xueyuan Avenue, Shenzhen, Guangdong, China}
\email{yut6@sustech.edu.cn}


\date{1 January 2026}

\begin{abstract}
  Baseilhac--Benedetti, following ideas of Kashaev, introduced invariants of
  pseudo-Anosov homeomorphisms of punctured hyperbolic surfaces that depend on a
  complex root of unity of odd order. Around the same time, Bonahon--Liu
  introduced another set of invariants of pseudo-Anosov homeomorphisms at roots of
  unity. A little later, Dimofte and the first
  author introduced invariants of cusped hyperbolic 3-manifolds at roots of unity
  using their geometric representation.
  In another effort, Bonahon--Wong--Yang introduced another set of invariants of
  pseudo-Anosov homeomorphisms at roots of unity. All these invariants are
  conjecturally closely related, and our aim is to prove a precise relation between
  the Baseilhac--Benedetti invariants, the Bonahon--Liu--Wong--Yang and the
  lesser-known abelian $\gl_1$-invariants.   
\end{abstract}

\maketitle

{\footnotesize
\tableofcontents
}


\section{Introduction}
\label{sec.intro}

\subsection{Invariants of pseudo-Anosov surface homeomorphisms}
\label{sub.intro}

Pseudo-Anosov (in short, pA) homeomorphisms of hyperbolic surfaces is a gem that
links two and
three-dimensional hyperbolic geometry, as was discovered by Thurston~\cite{Thurston}.
This connection leads to a concrete description of pA homeomorphisms and an explicit
determination of their invariants (see e.g.,~\cite{Farb}), such as their
two invariant singular foliations
on a surface, their stretch factors, the hyperbolic volume of their mapping torus,
and also a normal form of their conjugacy class expressed in terms of layered
triangulations~\cite{Agol}. What's more, all of these invariants are
explicitly computable (following the wish of Thurston) using standard software of
hyperbolic geometry such as \texttt{SnapPy} and \texttt{flipper}~\cite{snappy,flipper}.
In particular, the classification of pA classes comes together with easily accessible
examples illustrating the proven theorems.

So, in a sense, we know everything about pA surface homeomorphisms using hyperbolic
geometry. What more could we possibly want to know?
A partial answer to this question comes
from the problem of understanding topological quantum field theory (TQFT in short)
in 2+1 dimensions. The latter, among other things, assigns rich and
not-well-understood numerical invariants to 3-manifolds, and in particular to
mapping cylinders of pA surface homeomorphisms. These invariants typically depend
on a complex root of unity, and refinements of them suggest an extension of the
invariants using a background representation of the fundamental group to
$\SL_2(\BC)$. 

Baseilhac--Benedetti, following ideas of Kashaev~\cite{K95}, introduced invariants
of pseudo-Anosov homeomorphisms of punctured hyperbolic surfaces that depend on a
complex root of unity of odd order~\cite{BB2,BB:fiber}. Around the same time,
Bonahon--Liu introduced another set of invariants of pseudo-Anosov homeomorphisms
at roots of unity~\cite{BL}. A little later, Dimofte and the first
author introduced invariants of cusped hyperbolic 3-manifolds at roots of unity
using their geometric representation~\cite{DG2}.
In another effort, Bonahon--Wong--Yang introduced another set of invariants of
pseudo-Anosov homeomorphisms at roots of unity. All these invariants are
conjecturally closely related, and our aim is to prove a precise relation between
the Baseilhac--Benedetti (in short, BB) invariants, the Bonahon--Liu--Wong--Yang
(in short, BLWY) and the lesser-known abelian $\gl_1$-invariants.   

\subsection{Our results}
\label{sub.results}

Throughout the paper, we will work over the complex numbers $\cx$, and $q$ a complex
root of unity of odd order $n$. Note that $q^2$ was denoted by $\zeta$ in our
previous work~\cite{GY}.

The invariants of Baseilhac--Benedetti and \BLWY have similar origin.
They are both defined using an ideal triangulation $\lambda$ of a
punctured surface $\surface$ of genus $g$ with $p$ punctures of negative Euler
characteristic $\chi(\surface)=2-2g-p<0$. Such a triangulation gives rise
to two quantum tori (i.e, Laurent polynomial algebra of $q$-commuting operators),
one being a quotient $\qtorus^\CFr_q(\lambda)$ of the Chekhov--Fock algebra,
and another being the Kashaev algebra discussed in detail in
Sections~\ref{sub.kashaev} and~\ref{sub.CF} below.
A relation between the two is given by an embedding algebra map 
\begin{equation}
\label{3embed2}
\qtorus^\CFr_q(\lambda)\otimes_{\cx[P_\lambda]}\qtorus_q(H)
\embed\qtorus^\Kas_q(\lambda)
\end{equation}
of quantum tori (see Prop.~\ref{prop.3embed} below) where $\qtorus_q(H)$ is a
third quantum torus 
that depends on the homology $H=H_1(\surface,\BZ)$ of the surface and not on
the triangulation. The image of~\eqref{3embed2} is a quantum sub-torus of corank
$p-1$, where $p$ is the number of punctures of $\surface$.

An important point is that when $q$ is a root of unity, the representation theory
of a quantum torus is easy to describe, namely irreducible representations of
a quantum torus are induced by their restrictions to the center, a commutative
Laurent polynomial ring in several variables, and representation of the latter
are determined by their character. 

So far everything depends on an ideal triangulation $\lambda$ of $\surface$, but
any two such are related by a sequence of flips (and any two sequences are related
by a sequence of square and pentagon moves), and the quantum tori
$\qtorus^\CFr_q(\lambda)$ and $\qtorus^\Kas_q(\lambda)$ as well as the
comparison map~\eqref{3embed2} are related by isomorphisms. This generates an
action of the corresponding groupoid, and in particular it makes sense to talk about
characters $r$ that are invariant under a mapping class $\varphi$. 

The comparison map~\eqref{3embed2} decomposes a representation of
$\qtorus^\Kas_q(\lambda)$ into the tensor product of representations of
$\qtorus^\CFr_q(\lambda)$ and $\qtorus_q(H)$. This decomposition requires the use of
a generic character $r$ that admits a decoration $h$ such that $(r,h)$ is invariant
under the mapping class $\varphi$. Having done so, the decomposition expresses the
BB-invariants $ T^\Kas_{\varphi,r}(q)$ (Definition~\ref{def.K}) as a sum of the
BLWY-invariants $T^\CF_{\varphi,r,\zeta_v}(q)$ (Definition~\ref{def.CF}) times the
$\gl_1$-invariants $T^H_{\varphi,\zeta_v}(q)$ (Definition~\ref{def-gl1}). The latter
two depend additionally on puncture weights $\zeta_v\in\mu_n^p$%
\footnote{Here, $\mu_n$ denotes the set of all complex $n$-th roots of unity.} 
which assigns an $n$-th root of unity to each puncture such that the product is $1$.
The $\gl_1$-invariants are part of a TQFT in 2+1 dimensions based on the abelian
group $U(1)$ defined many years ago~\cite{MOO:gl1,Go:gl1}. Putting everything
together, we arrive at the following result.

\begin{theorem}
\label{thm.1}
With the above assumptions, we have
\begin{equation}
\label{inv-dec}
  \pm T^\Kas_{\varphi,r}(q)=\sum_{\zeta_v}\eta_{\zeta_v}T^\CF_{\varphi,r,\zeta_v}(q)
  T^H_{\varphi,\zeta_v}(q)
\end{equation}
where we sum over puncture weights $\zeta_v$ that are invariant under $\varphi$, and
$\eta_{\zeta_v} \in \mu_{n^{4g-4+2p}}$ are roots of unity to account for the
ambiguities of the invariants. The sign in \eqref{inv-dec} is the sign of the
puncture weight permutation induced by $\varphi$.
\end{theorem}

Some remarks are in order.

\begin{enumcc}[1.]
\item
  The identity of the above theorem is nontrivial since the representation of
  the mapping class group involved in $T^\Kas_{\varphi,r}(q)$ and
  $T^\CF_{\varphi,r,\zeta_v}(q)$ involves endomorphisms of vector spaces of dimension
  $n^{4g-4+2p}$ and $n^{3g-3+p}$, respectively.
\item
  Let $M_{\widehat\varphi}$ be the mapping torus of the mapping class
  $\widehat{\varphi}$ of $\surclose$ obtained by capping off all punctures. Then
  $\abs{T^H_{\varphi,\zeta_v}(q)}=\frac{1}{\sqrt{n}}
  \abs{H_1(M_{\widehat\varphi},\BZ/n\BZ)}^{1/2}$ or zero (see
  Corollary~\ref{cor-TH-values} below). When $\surface$ has only one puncture,
  then $\zeta_v=1$ and Equation~\eqref{inv-dec} simplifies to
\begin{equation}
\label{inv-dec2}
T^\Kas_{\varphi,r}(q)= \pm\eta_1T^\CF_{\varphi,r,1}(q) T^H_{\varphi,1}(q) \,.
\end{equation}
In addition, when $M_{\widehat\varphi}$ has the minimal $\BZ/n\BZ$-homology, that is,
$H_1(M_{\widehat\varphi},\ints/n\ints)=\ints/n\ints$, then $\abs{T^H_{\varphi,1}(q)}=1$, and
\begin{equation}
\label{inv-dec3}
\abs{T^\Kas_{\varphi,r}(q)} = \abs{T^\CF_{\varphi,r,1}(q)} \,,
\end{equation}
which is considerably different from a naive guess based on dimensions of the representations.
\item
  The above theorem involves numerical invariants of mapping classes. There is a
  stronger version that involves conjugacy classes of intertwiners; see
  Equation~\eqref{thm1-ops} below.
\item
  When $\varphi$ is a pA punctured surface homeomorphism, there is a canonical
  choice of $(r,h)$ corresponding to the geometric representation of the mapping
  torus. This is discussed in Section~\ref{sec.ge} below.
\item
  A computer implementation of the BB and BLWY invariants has been given by the
  second author and an illustration for pA classes in genus 1 and
  1 or 2 punctures is discussed in Section~\ref{sub.ex11} and \ref{sub.ex12} below.
\end{enumcc}

We end this section with a remark regarding the nature of the tensors involved in the
representations of the quantum tori discussed above.

\begin{remark}
\label{rem.etale}  
In our prior work we conjectured that the 1-loop invariants of~\cite{DG2} agree
with the BLWY invariants of a pA homeomorphism of a punctured surface using the
geometric $\PSL_2(\BC)$-representation~\cite[Conj.1.1]{GY}. We remark that the
1-loop invariants are simply the constant terms of power series around each root
of unity which can be arithmetically re-expanded from a root of unity of order $n$
to a root of unity of order $pn$, and determine one another up to
Frobenius twist~\cite{GSWZ}.

In that sense, the invariants at roots of unity are the \'etale realization
of the TQFT invariants of~\cite{GSWZ}.
\end{remark}

\subsection{Organization of the paper}
\label{sub.plan}

In Section~\ref{sec.qtorus} we discuss the basics on quantum tori and their
representation theory.

In Section~\ref{sec.invs} we introduce the quantum tori associated to punctured
surfaces with ideal triangulations, discuss connecting maps to pass from one
triangulation to another and define the BB and BLWY invariants. 

In Section~\ref{sec.proofs} we give the proofs of the statements made in
the previous section, completing the proof of Theorem~\ref{thm.1}.

In Section~\ref{sec.ge} we discuss the combinatorics of the layered ideal
triangulation of the mapping torus of a pA surface homeomorphism and their
gluing equations and use them to convert the endomorphism invariants of
a pA homeomorphism into the numerical BB and BLWY invariants. 


\section{Quantum tori}
\label{sec.qtorus}

\subsection{Definition}
\label{sub.qtorus} 

Suppose $L$ is a lattice equipped with an integer-valued skew-symmetric
bilinear form $\omega$.
Then the quantum torus $\qtorus_q$ associated with this data is the $\BC$-algebra
\begin{equation}
\qtorus_q(L)=\Span_\BC\{x^k\mid k\in L\},\qquad
x^kx^l=q^{\omega(k,l)}x^{k+l}.
\end{equation}
The quantum torus is graded by $L$, and subgroups of $L$ define subalgebras of
$\qtorus_q(L)$ that are themselves quantum tori. The rank (i.e., the Gelfand--Kirillov
dimension) of $\qtorus_q(L)$ is the rank of $L$.

In general, $\qtorus_q(L)$ is non-commutative. Its center, the coordinates ring
of a commutative torus, is given by $Z(\qtorus_q(L)) = \qtorus_q(L_n)$ where
\begin{equation}
L_n=\{k \in L \mid \omega(k,l) = 0\bmod{n},\; \text{for all }l\in L\} \,.
\end{equation}
If $\{\chi_i\}_{i \in I}$ is a basis of $L$, then we have a presentation
\begin{equation}
\qtorus_q(L)=\cx\langle x_i^{\pm1},i\in I\rangle/\ideal{x_ix_j-q^{2\omega_{ij}}x_jx_i},
\qquad
x_i=x^{\chi_i},
\end{equation}
where we use the shorthand $\omega_{ij}=\omega(\chi_i,\chi_j)$. Monomials are given
by the Weyl-ordered product
\begin{equation}
x^k=[x_1^{k_1}\dotsm x_n^{k_n}]:=q^{-\sum_{i<j}\omega_{ij}k_ik_j}x_1^{k_1}\dotsm x_n^{k_n}.
\end{equation}
The Weyl-ordering is defined more generally for $q$-commuting elements, and it is
commutative and associative in the sense that
\begin{equation}
[\dotsm ab\dotsm]=[\dotsm ba\dotsm],\qquad[a[bc]]=[[ab]c]=[abc]
\end{equation}
for $q$-commuting elements $a,b,c$.

Now suppose $(L,\omega),(L',\omega')$ are defining lattices of quantum tori. A
homomorphism $f:L'\to L$ such that $f^\ast\omega=\omega'$ defines an algebra
homomorphism between the quantum tori in the obvious way. In fact, we can insert
scalars and obtain a rescaled monomial map defined by
\begin{equation}
f_{\hat{s}\hat{s}'}:\qtorus_q(L')\to\qtorus_q(L),\qquad
f_{\hat{s}\hat{s}'}(x^{\prime k})=\frac{\hat{s}'(k)}{\hat{s}(f(k))}x^{f(k)},
\end{equation}
where $\hat{s}:L\to\cx^\times,\hat{s}':L'\to\cx^\times$ are group homomorphisms.

\subsection{Representations}
\label{sub.reps}

In this paper, all representations are finite dimensional $\BC$-vector spaces.

A main point of using quantum tori is that their representations are easy to
describe. Indeed, irreducible representations of commutative tori are given by
characters, and irreducible representations of quantum tori are uniquely determined
by their restriction to the center, a commutative torus. More precisely, it is
well-known that every finite dimensional irreducible representation
$\rho: \qtorus_q(L) \to \End(W)$ of $\qtorus_q(L)$
is (up to isomorphism) uniquely determined by its restriction to the center, and the
latter is uniquely determined by a character $s \in \Hom(L_n, \BC^\times)$ by Schur's
lemma. All such representations of the quantum torus have the same dimension
$D=\abs{L/L_n}^{1/2}$. Summarizing, we have a correspondence
\begin{equation}
\label{rhosig}
\rho: \qtorus_q(L) \to \End(W) \quad \leftrightarrow \quad s:
L_n \to \BC^\times, \qquad \dim(W)=\abs{L/L_n}^{1/2} \,.  
\end{equation}
To distinguish with the character variety of a surface, we call $s$ the central
character.

Given a central character $s$ as in~\eqref{rhosig}, there is a homomorphism
$\sigma$ defined by a composition
\begin{equation}
\label{sigma}
\sigma: L \to \BC^\times, \qquad
s: L \xrightarrow{\times n} n L \subset L_n \xrightarrow{s} \BC^\times
\end{equation}
This homomorphism is related to the Frobenius homomorphism. The latter 
is the embedding
\begin{equation}
\Phi_q:\qtorus_1(L)\embed\qtorus_q(L),\qquad \Phi_q(x^k)=x^{nk}.
\end{equation}
The image of $\Phi_q$ is the sub-torus $\qtorus_q(nL)$, which is contained in the
center. By definition, $\qtorus_1(L)$ is a commutative Laurent polynomial ring,
which is the coordinate ring of a classical algebraic torus. A point in the
classical torus is represented by an algebraic map $\qtorus_1(L)\to\cx$,
which is encoded by the group homomorphism $\sigma:L\to\cx^\times$ defined above.

The homomorphism $\sigma$ can be used to define the reduced quantum torus
at $\sigma$ by
\begin{equation}
\rdtorus_{q,\sigma}(L)=\qtorus_q(L)/\ideal{x^{nk}-\sigma(k)}.
\end{equation}
The reduced quantum torus (though not a quantum torus itself) will play an important
role in Section~\ref{sec-coord}.

Note that there is a one-to-one correspondence between representations
of $\rdtorus_{q,\sigma}(L)$ and representations of $\qtorus_q(L)$ whose central
characters $s$ are compatible with $\sigma$, i.e., satisfy $s(nk)=\sigma(k)$ for
all $k \in L$. We will freely switch between quantum
tori and their reductions depending on which is more convenient. Similarly, we
say a scalar map $\hat{s}:L\to\cx^\times$ extends $\sigma$ if $\hat{s}(nk)=\sigma(k)$.

\begin{remark}
A benefit of the reduction is that $1+q^{2j-1}x^k$ is invertible in
$\rdtorus_{q,\sigma}(L)$ if $\sigma(k)\ne-1$. This is used in Section~\ref{sec-coord}.
\end{remark}

We can find a basis $\{\alpha_1,\beta_1,\dotsc,\alpha_r,\beta_r,\gamma_1,
\dotsc,\gamma_{p-1}\}$ of $L$ such that $\omega$ has block diagonal form
$\bigoplus_{i=1}^rd_iJ_2\oplus(0)$ where $J_2=\begin{pmatrix}0&1\\-1&0\end{pmatrix}$
is the standard 2-dimensional symplectic form. Assume
\begin{equation}
\label{cond-even}
\text{all $d_i$ are powers of $2$.}
\end{equation}
Then $L_n=L_\infty+nL$, where $L_\infty$ is the radical of $\omega$, and the
dimension $D=n^r$.

A representation can be constructed explicitly given a basis as above. Let $V=\cx^n$
with the standard basis $\{e_j\}_{j\in\ints/n\ints}$. For later uses, let
$V_\ell=\cx e_\ell$.

Define operators
\begin{equation}
Se_j=q^{2j}e_j,\qquad Te_j=e_{j+1}.
\end{equation}
Given a central character $s:L_n\to\cx^\times$, choose an extension
$\hat{s}:L\to\cx^\times$; still a homomorphism, which exists since $L$ is a free
abelian group. Then the representation
\begin{equation}
\rho:\qtorus_q(L)\to\End(V^{\otimes r})
\end{equation}
is given by
\begin{equation}
\label{eq-qtorus-rep}
\begin{split}
\rho(x^{\alpha_i})&=\hat{s}(\alpha_i) 1\otimes\dotsm\otimes S^{d_i}\otimes\dotsm\otimes1,
\\
\rho(x^{\beta_i})&=\hat{s}(\beta_i) 1\otimes\dotsm\otimes T\otimes\dotsm\otimes1,
\\
\rho(x^{\gamma_i})&=\hat{s}(\gamma_i).
\end{split}
\end{equation}
It is easy to see that $\rho$ is the irreducible representation with central character
$s$.

The following lemmas are easy corollaries of the construction.

\begin{lemma}
\label{lem-mon-eig}
Let $\rho:\rdtorus_{q,\sigma}(L)\to\End(W)$ be an irreducible representation and
$x^k\in\qtorus_q(L)$ not in the center. Then the eigenvalues of $\rho(x^k)$ are all
the $n$-th roots of $\sigma(k)$ with equal multiplicity.
\end{lemma}

\begin{proof}
We can assume $k$ is primitive, in which case it can be part of a symplectic basis
with $\alpha_1=k$. Then the statement is clear from the construction
\eqref{eq-qtorus-rep}.
\end{proof}

\begin{lemma}
\label{lem-unitary}
For simplicity, suppose $\omega$ is symplectic. Let
$\rho:\rdtorus_{q,\sigma}(L)\to\End(V^{\otimes r})$ be the irreducible representation
defined by \eqref{eq-qtorus-rep}, and $f\in\Aut(\rdtorus_{q,\sigma}(L))$ be a
rescaled monomial map. Then $\rho$ and $\rho\circ f$ are unitarily equivalent.
\end{lemma}

\begin{proof}
The fact that they are isomorphic is due to the identical central character, a
condition hidden in the reduced quantum torus. Therefore, the nontrivial part
is the unitarity.

Suppose $\rho(x)=B^{-1}\cdot\rho(f(x))\cdot B$ for all $x\in\rdtorus_{q,\sigma}(L)$.
Then
\begin{equation}
v_k=B(e_{k_1}\otimes\dotsm\otimes e_{k_r})
\end{equation}
are simultaneous eigenvectors of $\rho(f(x^{\alpha_i}))$ with distinct simultaneous
eigenvalues $q^{2d_ik_i}\hat{s}(\alpha_i)$. Therefore, they are orthogonal with
respect to the standard Hermitian form on $V^{\otimes r}$.

On the other hand, a quick calculation shows that
\begin{equation}
  \rho(\hat{s}(l\cdot\beta)^{-1}f(x^{l\cdot\beta}))v_k=v_{k+l}\qquad
  \text{where }l\cdot\beta=\sum_{i=1}^r l_i\beta_i.
\end{equation}
Since $f$ is a rescaled monomial map,
$C=\rho(\hat{s}(l\cdot\beta)^{-1}f(x^{l\cdot\beta}))$ is the product of some
combination of $\rho(\hat{s}(\alpha_i)^{-1}x^{\alpha_i})$,
$\rho(\hat{s}(\beta_i)^{-1}x^{\beta_i})$ and some powers of $q$. This shows that
$C$ is unitary. In particular, all $v_k$ have the same norm.

Combining both parts, the columns of $B$ in the standard basis are orthogonal
with the same norm. By scaling the norm to be $1$, we obtain an intertwiner that
is unitary.
\end{proof}

\subsection{Restrictions to sub-tori}
\label{sub.res}

We look at a particular case of restrictions of irreducible representations of a
quantum torus to a sub-torus.

\begin{lemma}
\label{lem-sub-tori}
Let $(L,\omega)$ define the quantum torus $\qtorus_q(L)$ where $\omega$ is
nondegenerate and satisfies \eqref{cond-even}. Suppose $P\subset L$ is an isotropic
direct summand with orthogonal complement $P^\perp$ with respect to $\omega$
and $\rho:\qtorus_q(L)\to\End(W)$ is an irreducible representation
with central character $s:nL\to\cx^\times$. Then
\begin{equation}
W=\bigoplus_{u:P\to\cx^\times}W_u,\qquad
W_u=\{w\in W\mid \rho(x^k)w=u(k)w\text{ for }k\in P\}
\end{equation}
is the decomposition of $\rho$ into irreducible representations of
$\qtorus_q(P^\perp)$, where we require $u=s$ on $nP$.
\end{lemma}

\begin{proof}
It is easy to see that each $W_u$ is a representation of $\qtorus_q(P^\perp)$. To
show that they are irreducible, we use the explicit construction \eqref{eq-qtorus-rep}.
Choose a symplectic basis $\alpha_i,\beta_i$ of $L$ such that $P$ is spanned by a
subset of $\alpha_i$. We can assume $\rho$ is constructed using such a basis. In this
construction, $\qtorus_q(P)$ acts diagonally, and we can see that each $W_u$ is
nontrivial. Then a simple dimension count proves the decomposition.
\end{proof}

\begin{corollary}
The lemma still holds with $P^\perp$ replaced by a sublattice of $P^\perp$ with index
a power of $2$.
\end{corollary}

\begin{proof}
The dimension of irreducible representations does not change after passing to such
a sublattice.
\end{proof}


\section{Invariants of mapping classes}
\label{sec.invs}

\subsection{Surfaces and triangulations}
\label{sub.surf}

In this paper, we consider surfaces $\surface_{g,p}$ (abbreviated
by $\surface$) of genus $g$ with
a nonempty set $\punctures$ of $p$ punctures, with negative Euler characteristic
$\chi(\surface)=2-2g-p<0$. Such surfaces have an ideal triangulation $\lambda$
which is a set of edges that begin and end at the punctures such that each
connected component of $\surface \setminus \lambda$ is a triangle. It follows that
$\lambda$ consists of $-3\chi(\surface)$ edges and $\surface$ contains
$-2\chi(\surface)$ triangles.

There are two algebras, the Kashaev algebra and the Chekhov--Fock algebra
both being quantum tori, associated to such a surface and a triangulation.
In addition, there is a third $\gl_1$-algebra, also a quantum torus,
associated to the surface alone. We next introduce them, and discuss their relation. 

\subsection{The Kashaev algebra}
\label{sub.kashaev}

We present the algebra defined by Kashaev \cite{Kas} in stages. Write $F(\lambda)$
for the set of faces in the triangulation $\lambda$. For each face $t\in F(\lambda)$,
we assign a copy of $L_t=\ints^3/\ints\onevec$ equipped with a skew-symmetric form
induced by the form on $\ints^3$ with matrix
\begin{equation}
\begin{pmatrix}
0 & 1 & -1 \\
-1 & 0 & 1 \\
1 & -1 & 0
\end{pmatrix}.
\end{equation}
Here we use a bold-faced number to denote the vector whose components are all
that number.
The coordinate basis of $\ints^3$ are associated to the sides of $t$ in the
counterclockwise order. $L_t$ is isomorphic to $\ints^2$ with the standard symplectic
form, but we delay the choice of an isomorphism. Set
$L_\lambda=\bigoplus_{t\in F(\lambda)}L_t$ equipped with the direct sum form $\omega$.
The Kashaev algebra $\qtorus^\Kas_q(\lambda)$ is defined as the quantum torus
associated with $(L_\lambda,\omega)$.

\begin{remark}
The usual definition of the Kashaev algebra uses a decorated triangulation, which
numbers the faces and assigns a distinguished corner in each face by putting a dot.
In the face $t$, the edge opposite to the dot is eliminated so that we have an
isomorphism $L_t\cong\ints^2$ by $(x,y,z)\mapsto(y,-x)$. Then we get an isomorphism
\begin{equation}
L_\lambda=\bigoplus_{t\in F(\lambda)}L_t\cong\bigoplus_{i=1}^{\abs{F(\lambda)}}\ints^2
=\ints^{2\abs{F(\lambda)}},
\end{equation}
and the form $\omega$ is standard symplectic. The elementary move that change the
dots is an isomorphism between different Kashaev algebras in the usual definition,
whereas it is a different presentation of the same algebra in the definition of this
paper. The dots become important if we explicitly construct representations of the
Kashaev algebra using \eqref{eq-qtorus-rep}, which depends on the presentation of
the quantum torus.
\end{remark}

\subsection{The Chekhov--Fock algebra}
\label{sub.CF}

There is a skew-symmetric function $\varepsilon:\lambda\times\lambda\to\ints$
given by
\begin{equation}
\label{eq.Q}
\varepsilon_{ab} = \#\left( \facedef{$a$}{$b$} \right)
- \#\left( \facedef{$b$}{$a$} \right),
\end{equation}
where each shaded part is a corner of a face. We also use $\varepsilon$ to denote
the obvious extension as a bilinear form on $\ints^\lambda$. The Chekhov--Fock
algebra $\qtorus^\CF_q(\lambda)$ is the quantum torus associated to
$(\ints^\lambda,\varepsilon)$.
\begin{equation}
\qtorus^\CF_q(\lambda)=\cx\langle X_a^{\pm1},
a\in\lambda\rangle/\ideal{X_aX_b-q^{2\varepsilon_{ab}}X_bX_a}.
\end{equation}

For a puncture $v\in\punctures$, let $c_v\in\ints^\lambda$ be the vector that
counts how many times each edge ends on $v$. Then clearly
\begin{equation}
\sum_{v\in\punctures}c_v=\mathbf{2}.
\end{equation}

\begin{proposition}[{\cite[Proposition~5]{BL}}]
\label{prop-bl-pieces}
There exists a basis of $\ints^\lambda$ such that $\varepsilon$ is block-diagonal
with a $p$-dimensional block of $0$, $g$ blocks of $2J_2$, and $2g+p-3$ blocks of
$J_2$. The radical of $\varepsilon$, denoted $\ints^\lambda_\infty$ in the notation of
Section~\ref{sub.qtorus}, is spanned by $c_v$ defined above and $\onevec$. In
particular, $\varepsilon$ satisfies \eqref{cond-even}.
\end{proposition}

As we will see momentarily, the central element $X^\onevec$ is trivial when
considered together with the Kashaev algebra. We define the restricted Chekhov--Fock
algebra as the quotient
\begin{equation}
\label{CFr}  
\qtorus^\CFr_q(\lambda)=\qtorus^\CF_q(\lambda)/\ideal{X^\onevec-1}.
\end{equation}
It is the quantum torus based on the lattice $R_\lambda=\ints^\lambda/\ints\onevec$.

If the triangulation has self-folded faces, then it is sometimes convenient to
switch to a different generating set. All generators not on the folding edge are
kept the same. If a self-folded face has sides $i,i,b$, then the generator $X_i$ is
replaced with $X_iX_b$. With this modification, the Chekhov--Fock algebra becomes
part of a cluster algebra structure. See e.g.\ \cite{FG:q}.

\subsection{Relation between Kashaev and Chekhov--Fock algebras}
\label{sec-subalgs}

There are two important embedded subalgebras of the Kashaev algebra. The first one
is based on the homology group $H=H_1(\surface,\BZ)$ with twice the intersection form,
which satisfies \eqref{cond-even}. The embedding
\begin{equation}
i_H:H\embed L_\lambda
\end{equation}
sends a homology class $h\in H$ to the following construction. Represent $h$ as an
oriented simple multicurve on $\surface$ in minimal position with respect to the
triangulation $\lambda$. Each segment around a corner of a face is assigned the
generator on the side opposite to the corner. Then $i_H(h)$ is the signed sum of
these elements, where the signs are determined by whether the segments are clockwise
or counterclockwise. It follows from \cite{Kas} that this map is a well-defined
embedding and respects the forms.

The second one is the restricted Chekhov--Fock algebra. The map
\begin{equation}
i_\CFr:R_\lambda\embed L_\lambda
\end{equation}
sending each edge to the sum of the two sides in the adjacent faces is compatible
with the forms on the lattices essentially by definition. Clearly,
$\onevec\in\ints^\lambda$ is sent to $0\in L_\lambda$, so $i_\CFr$ is well-defined.
This is an embedding as we will see below.

Let $P_\lambda=\ints^\punctures/\ints\onevec$. It naturally embeds in both $R_\lambda$
and $H$. Specifically, $v\in P_\lambda$ is identified with $c_v\in R_\lambda$ and a
small circle $\gamma_v$ around $v$ in $H$. The two embeddings of $P_\lambda$
coincide in $L_\lambda$ after composition with $i_\CFr$ and $i_H$ respectively.
On the level of algebras, the group algebra $\cx[P_\lambda]$ embeds in the center
of both $\qtorus^\CFr_q(\lambda)$ and $\qtorus_q(H)$.

\begin{proposition}
\label{prop.3embed}
The lattice maps $i_H$ and $i_\CFr$ combine to induce an embedding of lattices
\begin{equation}
\label{latt-embed}
i:\left(R_\lambda\oplus H\right)/\ideal{c_v=\gamma_v}\embed L_\lambda
\end{equation}
and an associated embedding of quantum tori
\begin{equation}
\label{3embed}
i_q:\qtorus^\CFr_q(\lambda)\otimes_{\cx[P_\lambda]}\qtorus_q(H)
\embed\qtorus^\Kas_q(\lambda).
\end{equation}
\end{proposition}

\begin{proof}
It is known (e.g.\ \cite[Lemma~4]{Te:qTeich}) that the images of $i_\CFr$ and $i_H$
are orthogonal with respect to the symplectic form $\omega$ on $L_\lambda$ by a
simple calculation. Note each image contains a subgroup complementary to $P_\lambda$
such that the restriction of the form of $L_\lambda$ is non-degenerate. Then a simple
counting of ranks shows that $i$ is an embedding.
\end{proof}

We will omit the embeddings $i_H,i_\CFr,i,i_q$ if it is clear from context.

\begin{lemma}
\label{lem-sub-Kas}
$P_\lambda$ is a direct summand of $L_\lambda$, and $R_\lambda+H$ is a sublattice
of $P_\lambda^\perp$ of index $2^{2g}$.
\end{lemma}

\begin{proof}
To see the first part, we can construct a symplectic dual to $P_\lambda$ using
a maximal tree in $\lambda$. Choose $p-1$ punctures as generators of $P_\lambda$
and the remaining one as the base point of the maximal tree. There is a path from
the base point to each generator $v$, and form the sum of the left sides of the
path as an element of $b_v\in L_\lambda$. It is easy to check that
$\omega(b_v,c_w)=\delta_{vw}$.

The second part follows easily from a rank count and the block decompositions
of $\varepsilon$ and the intersection form on $H$.
\end{proof}

\subsection{Surface characters and central characters}
\label{sub.chars}

We consider decorated and framed character variety of $\surface$. For details,
see \cite{FG:mod}. 

Let $\FGA$ denote the decorated $\SL_2(\BC)$-character variety of $\surface$. 
Points in this space are decorated characters $(r,h)$ specified by a boundary
parabolic homomorphism
$r:\pi_1(\surface)\to\SL_2(\cx)$ and additional data $h$ related to punctures.
It is a complexified version of Penner's decorated Teichm\"uller space \cite{Pen}.

Given a peripheral subgroup $C\subset\pi_1(\surface)$, $r(C)$ has a unique fixed
point on $\partial_\infty\AMSbb{H}^3$. Given an ideal arc between puncture $a$ and
(possibly the same) puncture $b$, any peripheral subgroup $C_a$ around $a$
corresponds to a peripheral subgroup $C_b$ around $b$. We say $r$ is generic if for
any essential simple ideal arc, the fixed points associated to the ends are distinct.
Foe example, by \cite[Lemma~34]{BL}, $r$ is generic if $r$ is injective.

Given a triangulation $\lambda$, decorated generic characters can be parameterized
using Ptolemy coordinates $l_e$, $e\in\lambda$, which are complexified Penner's
length coordinates. Define the Kashaev coordinate of a side $a$ in the face with
sides $a,b,c$ in counterclockwise order as $l_b/l_c$. Thus, a decorated generic
character defines a homomorphism $\sigma:L_\lambda\to\cx^\times$. 
Summarizing, we have
\begin{equation}
\label{rhs}
(r,h) \in \FGA^\mathrm{generic} \mapsto \sigma \in \Hom(L_\lambda,\cx^\times) \,.
\end{equation}

Now consider the sub-tori. The homology part is trivial.

\begin{proposition}[Part of {\cite[Proposition~5]{Kas}}]\label{prop-Hchar}
The pullback $i_H^\ast\sigma\in\Hom(H,\cx^\times)$ is the trivial character.
\end{proposition}

By abuse of notations, we denote the pullback
$i_\CFr^\ast\sigma\in\Hom(R_\lambda,\cx^\times)$ also by $\sigma$. Then $\sigma$
describes generic and boundary parabolic points in the framed character variety
$\FGX$ using shear-bend coordinates. These points are determined by the projection of
$r$ to $\PSL_2(\cx)$ and are independent of the decorations. The generic condition
implies that shear-bend coordinates are never $-1$.

Since $(R_\lambda)_n=nR_\lambda+P_\lambda$, a central character
$s:(R_\lambda)_n\to\cx^\times$ of the restricted Chekhov--Fock algebra can be
determined by $\sigma$ and puncture weights $\zeta_v,v\in\punctures$ such that
\begin{equation}
\label{eq-weight-rel}
\prod_{v\in\punctures}\zeta_v=1,\qquad
\zeta_v^n=1,
\end{equation}
The puncture weights specify a map $P_\lambda\to\cx^\times$ compatible with $\sigma$.

\subsection{Coordinate change isomorphisms}
\label{sec-coord}

Traditionally, the coordinate change isomorphisms are defined by the same formulas
for all $q$ and without reduction, but in exchange, it is necessary to pass to some
localization. We also need to control the denominators for representations to exist,
which is nontrivial. To avoid such complication, we use reduced quantum tori, where
the inverses already exist.

Let $(r,h)$ be a decorated generic character. Given two triangulations
$\lambda,\lambda'$, there is a coordinate change isomorphism of algebras 
\begin{equation}
\label{theta}  
\Theta^{\lambda\lambda'}_q:\rdtorus^\Kas_{q,\sigma'}(\lambda')
\to\rdtorus^\Kas_{q,\sigma}(\lambda)
\end{equation}
where $\sigma:L_\lambda\to\cx^\times,\sigma':L_{\lambda'}\to\cx^\times$ are
coordinates of $(r,h)$ for the respective triangulations as in \eqref{rhs}.

It suffices to give the definition where $\lambda'=\lambda\setminus\{e\}\cup\{e'\}$
is a flip of $\lambda$. A general coordinate change is the composition of flips. The
local picture of a flip is given in Figure~\ref{fig-flip}. In this case,
$\Theta^{\lambda\lambda'}_q(x'_i)=x_i$ if $i$ is not in the local picture, and
\begin{equation}
\label{eq-flip-def}
\Theta^{\lambda\lambda'}_q(x'_i)=
\begin{cases}
x_i(1+qX_e),& i=a,c,\\
[x_iX_e](1+q^{-1}X_e)^{-1},& i=b,d,\\
x_bx_c,& i=e'_1,\\
x_ax_d,& i=e'_2.
\end{cases}
\end{equation}
Here, a prime is used to distinguish the generators of $\rdtorus^\Kas_{q,\sigma'}
(\lambda')$ from $\rdtorus^\Kas_{q,\sigma}(\lambda)$, and $X_e=x_{e_1}x_{e_2}$ is the
image of the generator from the Chekhov--Fock algebra. Note the cases $i=e'_1,e'_2$
are redundant.

These formulas are equivalent to those in \cite{Kas} without reduction. The fact
that they are still well-defined with reduction boils down to a simple calculation
showing the compatibility with the Frobenius homomorphism, that is,
\begin{equation}
\Theta^{\lambda\lambda'}_q\circ\Phi_q=\Phi_q\circ\Theta^{\lambda\lambda'}_1,
\end{equation}
and that $\Theta^{\lambda\lambda'}_1$ describes the classical coordinate change
relation between $\sigma$ and $\sigma'$.
As mentioned before, $1+q^{-1}X_e$ is invertible in the reduced quantum torus since
$t_e=\sigma(\chi_e)$ is the shear-bend coordinate of $e$, which is not $-1$ by
assumption.

\begin{figure}
\centering
\begin{tikzpicture}[baseline=(ref.base),inner sep=2pt]
\tikzmath{\s=0.1cm;\r=2;}
\begin{scope}[xshift=-\s,yshift=-\s]
\draw (0,0) -- node[above]{$b$} (\r,0) -- node[below left,inner sep=0pt]{$e_1$} (0,\r)
-- node[right]{$a$} cycle;
\end{scope}
\begin{scope}[xshift=\s,yshift=\s]
  \draw (\r,0) -- node[left]{$c$} (\r,\r) -- node[below]{$d$} (0,\r)
  -- node[above right,inner sep=0pt]{$e_2$} cycle;
\end{scope}
\tikzbase{1,1};
\end{tikzpicture}
\quad
$\longrightarrow$
\quad
\begin{tikzpicture}[baseline=(ref.base),inner sep=2pt]
\tikzmath{\s=0.1cm;\r=2;}
\begin{scope}[xshift=-\s,yshift=\s]
  \draw (0,0) -- node[above left,inner sep=0pt]{$e'_1$} (\r,\r)
  -- node[below]{$d$} (0,\r) -- node[right]{$a$} cycle;
\end{scope}
\begin{scope}[xshift=\s,yshift=-\s]
  \draw (\r,0) -- node[left]{$c$} (\r,\r)
  -- node[below right,inner sep=0pt]{$e'_2$} (0,0) -- node[above]{$b$} cycle;
\end{scope}
\tikzbase{1,1};
\end{tikzpicture}
\caption{Flip in the Kashaev algebra}
\label{fig-flip}
\end{figure}
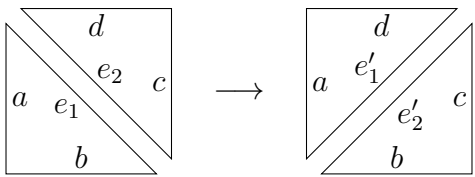

The coordinate change can be pulled back to the Chekhov--Fock algebra and the
homology quantum torus. The pullback to the homology part is just the identity. On
the Chekhov--Fock algebra, the pullback agrees with the standard formula of cluster
mutation. If there are no self-folded faces, then the formula is
\begin{equation}
\label{thetaq}
\Theta^{\lambda\lambda'}_q(X'_i)=
\begin{cases}
X_e^{-1},&i=e',\\
\displaystyle X_i\prod_{j=1}^{\abs{\varepsilon_{ie}}}(1+q^{2j-1}X_e),&
i\ne e',\,\,\varepsilon_{ie}\le0,\\
\displaystyle [X_iX_e^{\varepsilon_{ie}}]\prod_{j=1}^{\abs{\varepsilon_{ie}}}
(1+q^{1-2j}X_e)^{-1},&
i\ne e',\,\,\varepsilon_{ie}>0 \,.
\end{cases}
\end{equation}
If there are self-folded faces, we need to use the modified generating set as
described before.

\subsection{The BB and BLWY mapping class invariants}
\label{sub.inv}

In this section we define the BB and the BLWY invariants in operator and in
numerical form. We discuss the BB invariant first.

We fix a decorated generic character $(r,h)$ invariant under a mapping class
$\varphi$. Given a triangulation $\lambda$, let $\rho_\lambda:
\rdtorus^\Kas_{q,\sigma}(\lambda)\to\End(W)$ be the irreducible representation with
reduction $\sigma:L_\lambda\to\cx^\times$ determined by the decorated character
$(r,h)$ as in \eqref{rhs}.

There is an obvious lattice isomorphism $\varphi_\ast:L_\lambda\to L_{\varphi(\lambda)}$.
Then
\begin{equation}
\varphi_\ast\sigma:L_{\varphi(\lambda)}\cong L_\lambda\to\cx^\times
\end{equation}
corresponds to the same decorated character by the invariance. The composition
\begin{equation}
\label{rhoA}
\rdtorus^\Kas_{q,\sigma}(\lambda)\stackrel{\varphi_\ast}{\longrightarrow}
\rdtorus^\Kas_{q,\varphi_\ast\sigma} (\varphi(\lambda))
\stackrel{\Theta^{\lambda\varphi(\lambda)}_q}{\longrightarrow}
\rdtorus^\Kas_{q,\sigma}(\lambda)
\end{equation}
is an automorphism of $\rdtorus^\Kas_{q,\sigma}(\lambda)$. Therefore, we have another
representation $\rho_\lambda^\varphi=\rho_\lambda\circ
\Theta^{\lambda\varphi(\lambda)}_q\circ\varphi_\ast$ of the algebra
$\rdtorus^\Kas_{q,\sigma}(\lambda)$, and it is isomorphic to $\rho_\lambda$, so there
exists an intertwiner $A\in\GL(W)$ so that
\begin{equation}
\label{eq-inter-def}
\rho_\lambda(x)=A^{-1}\cdot\rho_\lambda^\varphi(x)\cdot A
\qquad\text{for }x\in\rdtorus^\Kas_{q,\sigma}(\lambda),
\end{equation}
unique up to scaling.

\begin{lemma}
The conjugacy and scaling equivalence class of the intertwiner does not depend on 
\begin{enumerate}
\item the choice of $\rho_\lambda$ within the isomorphism class,
\item the choice of $\lambda$, or
\item the decoration $h$.
\end{enumerate}
\end{lemma}

\begin{proof}
(1) is true essentially by definition. For (2), if $\lambda'$ is another
triangulation, then we can choose
$\rho_{\lambda'}=\rho_\lambda\circ\Theta^{\lambda\lambda'}_q$. It is easy to check
that $\rho_{\lambda'}^\varphi=\rho_\lambda^\varphi\circ\Theta^{\lambda\lambda'}_q$,
so the same intertwiner works.

(3) Any two decorations are related by a $(\cx^\times)^p$-scaling action. The induced
action on the Ptolemy coordinates is
\begin{equation}
(t_1,\dotsc,t_p)\cdot l_a=t_i t_j l_a,
\end{equation}
where $i,j$ are the two ends of the arc $a$. There is a further action on Kashaev's
coordinates since each of them is the ratio of two Ptolemy coordinates. Explicitly, if
$a$ is a side of a face such that the counterclockwise orientation of $a$ goes from
puncture $i$ to puncture $j$, then
\begin{equation}
(t\cdot\sigma)(\chi_a)=(t_j/t_i)\sigma(\chi_a).
\end{equation}
This action is covered by an action on the reduced Kashaev algebras
\begin{equation}
f_t:\rdtorus^\Kas_{q,\sigma}(\lambda)\to\rdtorus^\Kas_{q,t\cdot\sigma}(\lambda),
\qquad
f_t(x_a)=(t_j^{1/n}/t_i^{1/n})x_a.
\end{equation}
Note the Chekhov--Fock generators $X_e$ are invariant under this action. Observe that
the action $f_t$ is compatible with coordinate changes \eqref{eq-flip-def}. Then the
same argument as (2) shows that the intertwiner is independent of $h$.
\end{proof}

We now have all the ingredients to define the BB invariants.

\begin{definition}
\label{def.K}
Fix a generic character $r$ that admits a decoration $h$ such that $(r,h)$ is
invariant under a mapping class $\varphi$.
\begin{enumcc}[(a)]
\item We define
\begin{equation}
\label{AKdef}
A^\Kas_{\varphi,r}(q) \in \GL(W)
\end{equation}
to be the intertwiner in Equation~\eqref{eq-inter-def}, well-defined up to
conjugation and scalar multiplication.
\item We define $T^\Kas_{\varphi,r}(q)$ to be the trace of $A^\Kas_{\varphi,r}(q)$
after scaling so that the determinant is $1$. In other words,
\begin{equation}
\label{TKdef}
T^\Kas_{\varphi,r}(q)=\tr A^\Kas_{\varphi,r}(q)/\left(\det
  A^\Kas_{\varphi,r}(q)\right)^{1/n^{4g-4+2p}} \in \BC/\mu_{n^{4g-4+2p}} \,.
\end{equation}
\end{enumcc}
\end{definition}

For brevity, we often omit some data in the notation $A^\Kas_{\varphi,r}(q)$ if it is
clear from context. We will do the same for the other operator invariants defined
later.

Let us discuss some history and how the invariants defined above agree with the
Baseilhac--Benedetti invariants.

\begin{remark}
\label{rem.BB}  
\begin{enumcc}[1.]
\item
  Baseilhac--Benedetti (in short, BB) \cite{BB:fiber} considered so-called local
  representations of the Chekhov--Fock algebras. Local representations are a
  class of representations introduced by Bai--Bonahon--Liu \cite{Bai}. They are
  reducible, but better behaved under cut-and-paste. Due to the reducibility, BB
  cannot rely on Schur's lemma and gave explicit constructions of operators
  instead.
\item
  Local intertwiners are highly non-unique. Mazzoli \cite{Maz} constructed a
  subset of local intertwiners that form an $H_1(\surface,\ints/n\ints)$-torsor.
  BB's construction chooses one element in this subset and argues that it is
  canonical by 3-dimensional reasoning. Here, we effectively give a different
  reason using the Kashaev algebra.
\item
  Ishibashi \cite{Ishi} gave a formula for $A^\Kas_{\varphi,r}(q)$ using
  slightly different language, where the Kashaev algebra is called the Weyl
  algebra. Ishibashi further made the identification of the operator
  $A^\Kas_{\varphi,r}(q)$ with the reduced BB invariant. The difference between
  the reduced and the full BB invariants is the symmetry defect of
  \cite{BB:defect}. However, the symmetry defect associated to the layered
  triangulation of the mapping torus $M_\varphi$ is easily calculated to be $1$.
\item
  The trace $T^\Kas_{\varphi,r}(q)$ agrees with the earlier 3-dimensional BB
  invariant \cite{BB:analytic} of the mapping torus $M_\varphi$ by construction
  \cite{BB:fiber}. Note the 3-dimensional BB invariant depends on some choices of
  homology classes on $M_\varphi$, and the fibering implicitly chooses them.
\end{enumcc}
\end{remark}

The next remark is about the role of the decoration $h$.

\begin{remark}
An additional restriction that comes from the Kashaev algebra is that $r$ must admit
a decoration $h$, even though the invariant does not depend on it. In particular, $r$
must be boundary parabolic. There is a way to adjust the definitions here to work
with other generic characters, or one can use the construction of BB, which does not
require decorations to begin with. However, the generalization of our main
Theorem~\ref{thm.1} becomes more technical. We choose to omit the generalization for
clarity.
\end{remark}

The next remark is about the operator versus the numerical invariants of $\varphi$.

\begin{remark}
\label{rem.powers}
Note that 
\begin{equation}
A^\Kas_{\varphi,r}(q)^m = A^\Kas_{\varphi^m,r}(q), \qquad \text{for all} \,\, m \in \BZ \,.
\end{equation}
Hence, the characteristic polynomial of $A^\Kas_{\varphi,r}$ is determined by
$T^\Kas_{\varphi^m,r}$ for all integers $m$.
\end{remark}

This completes the definition of the BB invariants. We now discuss the BLWY invariants,
applying the same construction to the Chekhov--Fock algebra. Recall that a
Chekhov--Fock central character can be determined by $r$ and puncture weights $\zeta_v$
satisfying \eqref{eq-weight-rel}. Assuming these data are $\varphi$-invariant, then
the same process prior to Definition~\ref{def.K} produces the next definition.

\begin{definition}
\label{def.CF}
Fix a generic character $r$ that admits a decoration $h$ and puncture weights
$\zeta_v$ satisfying \eqref{eq-weight-rel}. Assume that $(r,h)$ is
invariant under a mapping class $\varphi$.
\begin{enumcc}[(a)]
\item We define
\begin{equation}
\label{ACFdef}
A^\CF_{\varphi,r,\zeta_v}(q) \in \GL(W)
\end{equation}
to be the intertwiner in the analog of Equation~\eqref{eq-inter-def} for the
Chekhov--Fock algebra, well-defined up to conjugation and scalar multiplication.
\item Taking the trace, we define
\begin{equation}
\label{TCFdef}
T^\CF_{\varphi,r,\zeta_v}(q)=
\tr A^\CF_{\varphi,r,\zeta_v}(q)/\left(\det A^\CF_{\varphi,r,\zeta_v}(q)\right)^{1/n^{3g-3+p}}
\in \BC/\mu_{n^{3g-3+p}}.
\end{equation}
\end{enumcc}
\end{definition}

This invariant was originally considered by \cite{BL}, which agrees with
\cite{BWY:I} under further generic conditions.

The next remark involves the root of unity ambiguity in the definition of
the numerical invariants from Equations~\eqref{TKdef} and \eqref{TCFdef}.

\begin{remark}
\label{rem.phase}
As mentioned in Remark~\ref{rem.BB}, BB defines operators by explicit formulas. The
advantage of such construction is that it reduces the phase ambiguity of the BB
operators (and hence of the corresponding numerical invariants) to $\mu_n$. Note
\cite{BB:fiber} only claims $\mu_{4n}$, but it cannot have a $\mu_4$ ambiguity
because of the determinant condition.

Similarly, the construction in \cite{Ishi} is given by explicit operators,
well-defined up to $\mu_{6n}$, which is also the ambiguity of the trace. However, the
determinant is not normalized to exactly $1$ there.

The apparent ambiguity of the definitions~\eqref{TKdef} and \eqref{TCFdef} are very
high. Since $T^\Kas_{\varphi,r}(q)$ is equal to BB invariant, there is a way to
reduce the ambiguity in~\eqref{TKdef} to an $n$-th root of unity. We expect the same
to be true for $T^\CF_{\varphi,r,\zeta_v}(q)$, possibly changing the normalization of
the determinant. Compare Remark~\ref{rem.gl1phase}.

The normalization by $\det=1$ is the simplest way to resolve the projective ambiguity
of the intertwiner. However, it may not be the natural one. This is evidenced by the
relation between $T^\CF_{\varphi,r,\zeta_v}(q)$ and the 1-loop invariants \cite{DG2}.
In \cite{GY}, we found that both the absolute value and the phase of
$T^\CF_{\varphi,r,\zeta_v}(q)$ needs to be adjusted to obtain a reasonable asymptotic
expansion confirming the quantum modularity conjecture.
\end{remark}

\subsection{The $\gl_1$-mapping class invariants}
\label{sec-qt-inter}

Finally, we consider invariants defined by representations of the quantum torus
$\qtorus_q(H)$ and identify them with the $\gl_1$-mapping class invariants. Recall
that $H=H_1(\surface,\BZ)$. 

By Proposition~\ref{prop-Hchar}, we only need to consider the reduced quantum torus
$\rdtorus_{q,1}(H)$ by the trivial character $H\to 1$. Then a central character of
$\rdtorus_{q,1}(H)$ is specified by puncture weights $\zeta_v$ only. Recall that the
product of $\zeta_v$ is required to be 1.

Number the punctures of $\surface$ by $1,\dotsc,p$, and write $\zeta_i=q^{\ell_i}$.
Let $\rho:\rdtorus_{q,1}(H)\to\End(W)$ be the irreducible representation constructed
by \eqref{eq-qtorus-rep} with puncture weights as above. We modify the representation
space of $\rho$ slightly to
\begin{equation}
W^H_{\zeta_v}=V^{\otimes g}\otimes(V_{\ell_1}\otimes\dotsm\otimes V_{\ell_p})
\end{equation}
to indicate the puncture weights. Correspondingly, the representation space of
$\rho\circ\varphi_\ast$ is $W^H_{\varphi(\zeta_v)}$. We are looking for an
intertwiner $A\in\Hom(W^H_{\zeta_v},W^H_{\varphi(\zeta_v)})$ such that
\begin{equation}
\label{eq-TQFT-inter}
\rho(x)=A^{-1}\cdot\rho(\varphi_\ast x)\cdot A,
\qquad x\in\rdtorus_{q,1}(H).
\end{equation}
A slight extension of \cite{Go:gl1,MOO:gl1} shows that the
$\gl_1$-Reshetikhin--Turaev operator invariant \cite{RT} for the mapping cylinder
\begin{equation}\label{eq-cylinder}
C_\varphi=(\surface\times[0,1]\sqcup\surface\times\{\ast\})/(a,0)\sim(f(a),\ast)
\end{equation}
is the intertwiner we want. We recall the definition below.

Let $G_0$ be the ribbon graph with $g$ handles and $p$ loose ends as shown in
Figure~\ref{fig-Vgp}, and $G_1$ be the mirror image. The loose ends of $G_i$
corresponds to the punctures of the surface. The boundary of a neighborhood of $G_i$
is identified with the surface $\surface$, where the $\alpha$ curves are the
meridians of the handles by convention.

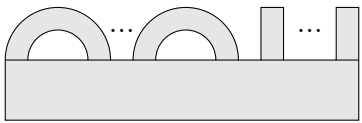
\begin{figure}
\centering
\begin{tikzpicture}
\tikzmath{\w=0.3;\d=0.8;\r1=\d/2;\r2=\r1+\w;\h=0.8;}
\draw[fill=gray!20,even odd rule] (0,0) -- (0,\h)
  arc[radius=\r2,start angle=180,end angle=0] -- ++(\w,0)
  arc[radius=\r2,start angle=180,end angle=0] -- ++(\w,0)
  -- +(0,\r2) -| ++(\w,0) -- ++(\r2,0) |- ++(\w,\r2) |- cycle
  (\w,\h) arc[radius=\r1,start angle=180,end angle=0] -- cycle
  (\d+4*\w,\h) arc[radius=\r1,start angle=180,end angle=0] -- cycle;
\foreach \x in {0,\w+\d,3*\w+\d,4*\w+2*\d,6*\w+2*\d,7*\w+2*\d+\r2}
\draw (\x,\h) -- +(\w,0);
\foreach \x in {2.5*\w+\d, 7*\w+2*\d+\r2/2}
\path (\x,\h+\r1)node{...};
\end{tikzpicture}
\caption{Ribbon graph $G_0$}
\label{fig-Vgp}
\end{figure}

Let $G$ be the union $G_0\cup G_1$ where the loose ends are connected according to
the action of $\varphi$. Then $C_\varphi$ can be represented by a diagram $L\sqcup
G\subset S^3$ where $L$ is a framed link used for surgery, and a small neighborhood
of $G$ is removed to produce $\partial C_\varphi$. As an example, the identity
mapping class is represented by Figure~\ref{fig-id-surg} where all curves have
blackboard framing. A Dehn twist is given by an extra $\pm1$-surgery on the core of
the twist. In general, there is a surgery diagram given by changing the braiding of
the punctures and inserting $\pm1$-framed curves (Dehn twists) near $G_0$ since they
generate the mapping class group. Therefore, we can assume that the entire diagram
lies in $\AMSbb{R}^2\times[0,1]$ such that the only parts touching the boundary are
the top and bottom coupons.

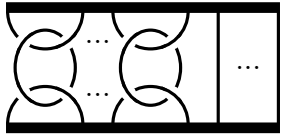
\begin{figure}
\centering
\begin{tikzpicture}[very thick]
\tikzmath{\w=0.4;\r=0.5;\h=1.5;\l=4*\r+4*\w;\a=2*\r+\w/2;}
\begin{knot}[background color=white]
\strand (0,0) arc[radius=\r,start angle=180,end angle=0]
  ++(\w,0) arc[radius=\r,start angle=180,end angle=0];
\strand (0,\h) arc[radius=\r,start angle=-180,end angle=0]
  ++(\w,0) arc[radius=\r,start angle=-180,end angle=0];
  \strand (\r,\h/2) circle[x radius=0.8*\r,y radius=\r] (3*\r+\w,\h/2)
  circle[x radius=0.8*\r,y radius=\r];
\flipcrossings{1,3,5,7}
\end{knot}
\foreach \x in {\l,\l-2*\w} \draw (\x,0) -- +(0,\h);
\draw[fill] (0,0) rectangle (\l,-0.1) (0,\h) rectangle (\l,\h+0.1);
\path (\a,0.8*\r)node{...} (\a,\h-0.8*\r)node{...} (\l-\w,\h/2)node{...};
\end{tikzpicture}
\caption{Surgery diagram for the identity mapping class}\label{fig-id-surg}
\end{figure}

The braided Hopf algebra used for this RT invariant is $\cx[K]/\ideal{K^n-1}$ with
braiding $R=\frac{1}{n}\sum_{i,j=0}^{n-1}q^{ij}K^i\otimes K^j$.
The irreducible representations of this algebra are identified with $V_\ell$ where
$K$ acts as $q^\ell$.
Then by construction, $W^H_{\zeta_v}=V^{\otimes
g}\otimes(V_{\ell_1}\otimes\dotsm\otimes V_{\ell_p})$ is the TQFT space of
$\surface\times\{1\}$ with the punctures colored by $\ell_1,\dotsc,\ell_p$, and
similarly, $W^H_{\varphi(\zeta_v)}$ is the TQFT space of $\surface\times\{\ast\}$.

\cite[Section~8]{MOO:gl1} gives an explicit formula for the RT operator invariant of
$C_\varphi$ when the surface is closed. We can generalize to the punctured case as
follows. A coloring of $L\sqcup G$ is an assignment of numbers in $\ints/n\ints$ to
each component of $L$ and each ribbon of $G$. They are assembled into a vector
$c=(k,h_0,h_1,\ell_1,\dots,\ell_p)$, where $k,h_0,h_1$ are the colors of $L,G_0,G_1$
respectively.

\begin{definition}\label{def-gl1}
Let $Q$ be the linking matrix of $L\sqcup G$ with the rows and columns ordered
according to the entries of coloring vectors $c$. Define
\begin{equation}\label{eq-AH-def}
A^H_{\varphi,\zeta_v}(q):e_{h_1}\mapsto\frac{\delta(q)^{-\sigma(Q)}}{\sqrt{n^{g+\#L}}}
\sum_{h_0\in(\ints/n\ints)^g}\bigg(\sum_{k\in(\ints/n\ints)^{\#L}}q^{c^tQc}\bigg) e_{h_0}.
\end{equation}
Here, $\delta(q)=\frac{1}{\sqrt{n}}\sum_{i\in\ints/n\ints}q^{i^2}$ is the phase of
the Gauss sum, $\sigma$ denotes the signature of the symmetric matrix $Q$,
$e_{h_1}\in W^H_{\zeta_v}$ and $e_{h_0}\in W^H_{\varphi(\zeta_v)}$ are standard basis
vectors. Note the puncture colors $\ell_1,\dotsc,\ell_p$ are not summed.

As before, we define
\begin{equation}
\label{THdef}
T^H_{\varphi,\zeta_v}(q)=\tr A^H_{\varphi,\zeta_v}(q)/
\left(\det A^H_{\varphi,\zeta_v}(q)\right)^{1/n^g}\in\cx/\mu_{n^g}.
\end{equation}
\end{definition}

\begin{remark}
Since the ribbon graph gives a framing to the punctures, technically we are
considering the mapping class group with framed punctures, which differs from the
usual one by the Dehn twists around the punctures. This only changes the framings of
the puncture braids, which changes $A^H_{\varphi,\zeta_v}(q)$ by $\mu_n$.

For the rest of the section, we work with framed lifts of the mapping class,
which implies the result for unframed ones with an extra $\mu_n$ ambiguity.
\end{remark}

\begin{theorem}
\label{thm-gl1-RT}
Suppose all mapping class are framed as in the remark above.
\begin{enumerate}
\item
  $A^H_{\varphi,\zeta_v}(q)$ is well-defined (with no root of unity ambiguity).
\item
  $A^H_{\varphi\circ\psi,\zeta_v}(q)=A^H_{\varphi,\psi(\zeta_v)}(q)\circ A^H_{\psi,\zeta_v}(q)$
  up to $\mu_4$.
\item
  In the standard bases of $W^H_{\zeta_v},W^H_{\varphi(\zeta_v)}$, we have
  $\det\left(A^H_{\varphi,\zeta_v}(q)\right) \in \mu_{12}$ if $g\ge1$.
\item
  $A^H_{\varphi,\zeta_v}(q)$ is an intertwiner of $\rho$ and $\rho\circ\varphi_\ast$.
\end{enumerate}
\end{theorem}

Note if $g=0$, then $W,W'$ are 1-dimensional, and $A_\varphi$ is a power of $q$.

\begin{proof}
(1) and (2) are straightforward generalizations of Theorem~1.3 and Proposition~8.1
of \cite{MOO:gl1}, except we assume that $n$ is odd so that the phase of the Gauss
sum is a 4th root instead of an 8th root.

For (3) and (4), it suffices to check the generators of the mapping class group by
(2). Most generators are given in \cite[Section~8]{MOO:gl1}, so we just need to add
the Dehn twists given in Figure~\ref{fig-mcg-gens} and braids of the punctures. Here,
only the modification compared to the identity class in Figure~\ref{fig-id-surg} is
shown, where the framing of the blue curve is $\pm1$. The results follow from a case
by case calculation.
\end{proof}

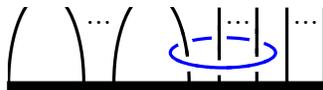
\begin{figure}
\centering
\begin{tikzpicture}[very thick]
\tikzmath{\w=0.4;\r=0.5;\h=1;\b=4*\r+2*\w;\l=\b+2*\r+\w;}
\begin{scope}
\clip (0,0) rectangle (\l,\h);
\begin{knot}[background color=white]
\strand[x radius=\r,y radius=\h+0.2] (0,0) arc[start angle=180,end angle=0]
  ++(\w,0) arc[start angle=180,end angle=0]
  (\b,0) -- +(0,\h) (\b+\r,0) -- +(0,\h);
\strand[blue] (\b+0.05,0.4)circle[x radius=1.4*\r,y radius=0.2];
\flipcrossings{2,4,6}
\end{knot}
\end{scope}
\foreach \x in {\l,\l-\r} \draw (\x,0) -- +(0,\h);
\draw[fill] (0,0) rectangle (\l,-0.1);
\foreach \x in {2*\r+\w/2,\l-\r/2,\l-\w-1.5*\r}
\path (\x,0.8*\h)node{...};
\end{tikzpicture}
\caption{Additional Dehn twist generators of the mapping class group}
\label{fig-mcg-gens}
\end{figure}

For a closed 3-manifold $M$, \cite{MOO:gl1} gave the formula for the numerical
$\gl_1$-RT invariants, which can be generalized to include an embedded colored
$\gl_1$-link $U$, except we divide their formula by $\sqrt{n}$ so that $S^2\times
S^1$ has invariant $1$, as opposed to $\sqrt{n}$ in \cite{MOO:gl1}. Denote this
generalized and renormalized invariant by $Z_q(M,U)$.

\begin{theorem}
\label{thm.gl1}
Suppose $\varphi$ is framed. Then $\tr A^H_{\varphi,\zeta_v}(q)$ agrees with
$Z_q(M_{\widehat\varphi},\widehat{P}_\varphi)$ up to an unambiguous $4$-th root of
unity. Here, $M_{\widehat\varphi}$ is the closed mapping torus of the mapping class
$\widehat{\varphi}$ of $\surclose$ obtained by capping off all punctures,
$\widehat{P}_\varphi$ is the colored link in $M_{\widehat\varphi}$ obtained by the
closure of the puncture braid.

Consequently, $T^H_{\varphi,\zeta_v}(q)$ agrees with
$Z_q(M_{\widehat\varphi},\widehat{P}_\varphi)$ up to $\mu_{12n^g}$.
\end{theorem}

\begin{remark}
\label{rem.gl1phase}
We emphasize that both the operator $A^H_{\varphi,\zeta_v}(q)$ and the numerical
invariant $Z_q(M_{\widehat\varphi},\widehat{P}_\varphi)$ have no ambiguity once a
framed lift of $\varphi$ is chosen. On the other hand, $\det
A^H_{\varphi,\zeta_v}(q)$ is not $1$ in general, but we need to normalize it to be
exactly $1$ for the main Theorem~\ref{thm.1} in its current form, which is why
$T^H_{\varphi,\zeta_v}(q)$ has a $\mu_{n^g}$ ambiguity. 
\end{remark}

\begin{proof}
A surgery diagram of $(M_{\widehat\varphi},\widehat{P}_\varphi)$ can be obtained by the
closure of the surgery diagram $L\sqcup G$ with one extra $0$-framed belt loop. See
Figure~\ref{fig-mtorus-surg}. The belt loop has $0$ linking with all other summed
components. In addition, recall the product of the puncture weights is assumed to be
$1$. This shows that the sum over the belt loop is completely decoupled, giving a
factor of $n$, cancelled by a $\sqrt{n}$ in the prefactor and another $\sqrt{n}$ by
our normalization choice. The rest of the sum is exactly the definition of trace. The
difference of $\mu_4$ comes from the difference in the signature of the linking
matrices.
\end{proof}

\begin{figure}
\centering
\begin{tikzpicture}[very thick]
\tikzmath{\w=0.4;\r=0.5;\h=1.2;\l=5*\w;\a=3*\w;}
\draw (0,0) arc[x radius=\w,y radius=\r,start angle=180,end angle=0]
  (0,\h) arc[x radius=\w,y radius=\r,start angle=-180,end angle=0];
\foreach \x in {\l,\l-\w} \draw (\x,0) -- +(0,\h);
\foreach \y in {0.25*\r,\h-0.25*\r} \path (\a,\y)node{...};
\draw[fill] (0,0) rectangle (\l,-0.1) (0,\h) rectangle (\l,\h+0.1);
\fill[gray!60] (-0.1,0.4*\r) rectangle (\l+0.1,\h-0.4*\r);
\path (\l/2,\h/2)node{$L\sqcup G$};
\draw[->] (\l+0.5,\h/2) -- +(1,0);
\begin{scope}[xshift={(\l+2)*1cm}]
\draw (0,0) arc[x radius=\w,y radius=\r,start angle=180,end angle=0]
  (0,\h) arc[x radius=\w,y radius=\r,start angle=-180,end angle=0];
\foreach \x in {\l,\l-\w} \draw (\x,0) -- +(0,\h);
\foreach \y in {0.25*\r,\h-0.25*\r} \path (\a,\y)node{...};
\begin{knot}[background color=white]
\strand[x radius=\l-(\x)+\w,y radius=(\l-(\x)+\w)/2] foreach \x in {0,2*\w,\l-\w,\l}
{(\x,0) arc[start angle=-180,end angle=0] -- ++(0,\h) arc[start angle=0,end angle=180]};
\strand[red] (1.5*\l+2*\w,\h/2) circle[x radius=\l/2+\w,y radius=\h/4];
\flipcrossings{2,4,6,8}
\end{knot}
\fill[gray!60] (-0.1,0.4*\r) rectangle (\l+0.1,\h-0.4*\r);
\path (\l/2,\h/2)node{$L\sqcup G$};
\end{scope}
\end{tikzpicture}
\caption{Surgery diagram for the mapping torus}\label{fig-mtorus-surg}
\end{figure}
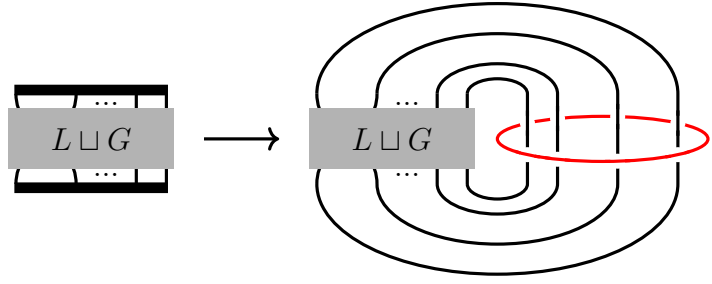

\begin{corollary}
\label{cor-TH-values}
\begin{enumcc}[(a)]
\item
The only possible values of $\abs{T^H_{\varphi,\zeta_v}(q)}$ are
$\frac{1}{\sqrt{n}}\abs{H_1(M_{\widehat\varphi},\ints/n\ints)}^{1/2}$ or $0$.
The former is achieved when $\zeta_v=1$.
\item
If $H_1(M_{\widehat\varphi},\ints/n\ints)=\ints/n\ints$,
then $\abs{T^H_{\varphi,\zeta_v}(q)}=1$ for all puncture weights.
\end{enumcc}
\end{corollary}

\begin{proof}
(a) First, if all puncture weights are $1$, then $\widehat{P}_\varphi$ does not
contribute to the terms, so the $\gl_1$-RT invariant is given in \cite{MOO:gl1},
whose renormalization gives the formula above.

Since the $\gl_1$-RT invariant is given by a quadratic sum of the form
\begin{equation}
\sum_{k\in(\ints/n\ints)^{\#L}}q^{k^tQk+2\ell^tk}
\end{equation}
with some fixed prefactor, there are only two possibilities. Either the linear term
can be absorbed by a substitution, or there is a flat direction of $Q$ over
$\ints/n\ints$ where the linear term is nontrivial. In the first case, the sum is
independent of the puncture weights, so the formula for $\zeta_v=1$ applies. In the
second case, the sum in the flat direction gives $0$.

(2) By \cite[Lemma~3.3]{MOO:gl1}, after removing the belt loop, $Q$ is nondegenerate
over $\ints/n\ints$. Therefore, all puncture weights fall into the first case.
\end{proof}

\subsection{Proof of the main theorem}
\label{sub.thm1}

In this section we reduce the proof of the main theorem~\ref{thm.1} to the proof
of a Lemma~\ref{lem-key}.

Fix a generic character $r$ that admits a decoration $h$ and puncture weights
$\zeta_v$ satisfying \eqref{eq-weight-rel}. Assume that $(r,h)$ is
invariant under a mapping class $\varphi$.
Then, we obtain a representation
\begin{equation}
\rho_\lambda:\rdtorus^\Kas_{q,\sigma}(\lambda)\to\End(W) \,.
\end{equation}
Using Lemmas~\ref{lem-sub-tori} and \ref{lem-sub-Kas}, $\rho_\lambda$ decomposes
into irreducible representations 
\begin{equation}
W=\bigoplus_{\zeta_v}W_{\zeta_v},\qquad
W_{\zeta_v}=\{w\in W\mid \rho(X^{c_v})w=\zeta_vw\text{ for }v\in\punctures\}
\end{equation}
of $\rdtorus^\CFr_{q,\sigma}(\lambda)\otimes_{\cx[P_\lambda]}
\rdtorus_{q,1}(H)$.

The action of $\varphi_\ast$ permutes these components, and
$\Theta^{\lambda\varphi(\lambda)}_q$ preserves the components. Therefore, the
intertwiner $A^\Kas_{\varphi,r}(q) \in\GL(W)$ (abbreviated by $A^\Kas$ in the rest of
the proof) for the Kashaev algebra has a block decomposition
\begin{equation}
\label{eq-interK-block}
A^\Kas = \bigoplus_{\zeta_v} A^\Kas_{\zeta_v}, \qquad
A^\Kas_{\zeta_v}=A^\Kas
|_{W_{\zeta_v}}:W_{\zeta_v}\to W_{\varphi(\zeta_v)}
\end{equation}
which are intertwiners of the components
\begin{equation}
  \rho_\lambda(x)|_{W_{\zeta_v}}=(A^\Kas_{\zeta_v})^{-1}\cdot
  \rho_\lambda^\varphi(x)
  |_{W_{\varphi(\zeta_v)}}\cdot A^\Kas_{\zeta_v},
\qquad x\in\rdtorus^\CFr_{q,\sigma}(\lambda)\otimes_{\cx[P_\lambda]}\rdtorus_{q,1}(H).
\end{equation}
Here, $\varphi(\zeta_v)$ denotes the set of puncture weights which is $\zeta_v$ on
$\varphi(v)$.

\begin{lemma}
\label{lem-key}
With the notations above, there exists a choice of $A^\Kas$ and choices of bases
of $W_{\zeta_v}$ such that $\det(A_{\zeta_v})=1$ for all $\zeta_v$.
\end{lemma}

This is a key lemma and an entire next Section~\ref{sec.proofs} is devoted to it.
Assuming this Lemma, we can now give a proof of Theorem~\ref{thm.1}. 

\begin{proof}[Proof of Theorem~\ref{thm.1}]
Each component $W_{\zeta_v}$ factors as a tensor product of irreducible representations
\begin{equation}
W_{\zeta_v}\cong W^\CF_{\zeta_v}\otimes_\cx W^H_{\zeta_v},
\end{equation}
where $W^\CF_{\zeta_v}$ and $W^H_{\zeta_v}$ are irreducible representations of
$\rdtorus^\CFr_{q,\sigma}(\lambda)$ and $\rdtorus_{q,1}(H)$, respectively. Both
$\varphi_\ast$ and $\Theta^{\lambda\varphi(\lambda)}_q$ respect the tensor product
$\rdtorus^\CFr_{q,\sigma}(\lambda)\otimes_{\cx[P_\lambda]}\rdtorus_{q,1}(H)$. Thus,
the intertwiner of the invariant blocks decomposes into
\begin{equation}
\label{operators}  
A^\Kas_{\zeta_v}=A^\CF_{\zeta_v}\otimes_\cx A^H_{\zeta_v},\qquad
A^\bullet_{\zeta_v}\in\Hom(W^\bullet_{\zeta_v},W^\bullet_{\varphi(\zeta_v)}).
\end{equation}
In other words,
\begin{equation}
\label{thm1-ops}
A^\Kas = \bigoplus_{\zeta_v} A^\CF_{\zeta_v}\otimes_\cx A^H_{\zeta_v}.
\end{equation}
This proves a stronger operator version of Theorem~\ref{thm.1}.

To obtain Theorem~\ref{thm.1}, we choose $A^\Kas$ according to the lemma. Each block
has determinant $1$, so we can assume that $A^\bullet_{\zeta_v}$ have determinant $1$
as well. On the other hand, we only have $\det A^\Kas=\pm1$ since the blocks permute
the puncture weights, so $T^\Kas_{\varphi,r}(q)=\pm\tr A^\Kas$. On the other side,
only the components fixed by $\varphi$ contribute to the trace, and clearly
$\tr(A^\CF_{\zeta_v}\otimes_\cx
A^H_{\zeta_v})=T^\CF_{\varphi,r,\zeta_v}(q)T^H_{\varphi,\zeta_v}(q)$ for the ones
that contribute. This proves the main theorem.
\end{proof}

\subsection{Further decompositions of representations of the Chekhov--Fock algebra}

In this section we comment on two further invariants that BB and Ishibashi defined,
and in fact we will see that they are not new. We start with an irreducible
representation $W$ of the Kashaev algebra. 
As mentioned in Remark~\ref{rem.BB}, previous works consider $W$ as a reducible
representation of the Chekhov--Fock algebra only. We discuss two more definitions
related to the decomposition of $W$.

We continue to use the notations in the last section. By \cite{Tou}, there is a
decomposition
\begin{equation}
W\cong\bigoplus_{\zeta_v}n^gW^\CF_{\zeta_v}
\end{equation}
as representations of the Chekhov--Fock algebra. The decomposition is not
unique, but each summand $n^gW^\CF_{\zeta_v}$, namely the isotypical components,
is well-defined. It is clear that these components are simply $W_{\zeta_v}$.
Therefore, BB's isotypical intertwiner $L^\varphi_{\rho_\lambda(\mu)}$ in
\cite{BB:fiber} is $A^\Kas_{\zeta_v}$ in our notation. If the puncture weights
are invariant, then the trace of the isotypical intertwiner is
$T^\CF_{\varphi,r,\zeta_v}(q)T^H_{\varphi,\zeta_v}(q)$, as discussed at the end
of the last section.

Ishibashi constructed a specific decomposition of $W$ into irreducible components in
\cite{Ishi}. The components are described as simultaneous eigenspaces for
$\rdtorus_{q,1}(I)$ with $I\subset H$ maximal isotropic. Explicitly, given
$c:I\to\mu_n$,
\begin{equation}
W_c=\{w\in W\mid\rho_\lambda(x^k)w=c(k)w\text{ for all }k\in I\}
\cong W^\CF_{\zeta_v}\otimes_\cx \Span\{w_c\},
\end{equation}
where $w_c\in W^H_{\zeta_v}$ is a simultaneous eigenvector of
$\rdtorus_{q,\sigma}(I)$. Note the puncture weights are included in $c$.

In general, the components $W_c$ are permuted by $A^\Kas$. However, the permutation
can be cancelled by some further choices. Let $B\in\Sp(H)$ such that
$\varphi_\ast\circ B$ restricts to identity on $I$. Let
$B^H_{\zeta_v}\in\GL(W^H_{\zeta_v})$ be the intertwiner associated to the monomial
automorphism of $\rdtorus_{q,1}(H)$ induced by $B$. Using the construction in
Lemma~\ref{lem-unitary}, we see that $A^H_{\zeta_v}B^H_{\zeta_v}=\diag\{q^{\ell_c}\}$
is diagonal in the basis $\{w_c\}$. Therefore,
\begin{equation}
  A^\Kas_{\zeta_v}\circ(1\otimes B^H_{\zeta_v})=A^\CF_{\zeta_v}
  \otimes_\cx(A^H_{\zeta_v}B^H_{\zeta_v})=\bigoplus_c A^\CF_{\zeta_v}\otimes q^{\ell_c}
\end{equation}
This proves \cite[Conjecture~8.6]{Ishi} under the conditions of Theorem~\ref{thm.1}.

\begin{remark}
\label{fixedp}
We remark that both BB and Ishibashi implicitly assumed that the puncture
weights are fixed by the mapping class. In our setting, this does not have to be
the case; see Example $\varphi=T_b^{-1} T_a T_c X$ of Section~\ref{sub.ex12} on
a surface of genus one with two punctures.
\end{remark}


\section{Operator determinants are roots of unity}
\label{sec.proofs}

This section is devoted to the proof of Lemma~\ref{lem-key}.

\subsection{Factorization of the flip}

In this section we recall a factorization of $\Theta^{\lambda\lambda'}_q$ for a flip
~\eqref{theta} of two triangulations $\lambda$ and $\lambda'$ of the surface. This
factorization is given by Fock--Goncharov \cite{FG:q} in terms of a formal adjoint
action. We discuss a modified version of this factorization with a concrete adjoint.

Continuing the notations from Section~\ref{sec-coord}, let
$\hat{s}':L_{\lambda'}\to\cx^\times$ be an extension of $\sigma'$,
and choose an $n$-th root $\theta=(1+t_e)^{-1/n}$, where $e$ is the edge getting
flipped. 

By convention, the adjoint is given
by
\begin{equation}
\Ad(A)(B)=ABA^{-1}.
\end{equation}
There is a factorization of the coordinate change
\begin{equation}
\label{eq-flip-factor}
\Theta^{\lambda\lambda'}_q=\Ad(\Psi_q^\theta(X_e))\circ
m^{\lambda\lambda'}_{\hat{s}\hat{s}^\theta}
\end{equation}
as follows. Define a linear map $m^{\lambda\lambda'}:L_{\lambda'}\to L_\lambda$ by
$m^{\lambda\lambda'}(\chi'_i)=\chi_i$ if $i$ is not in the local picture
Figure~\ref{fig-flip}, and
\begin{equation}
\begin{aligned}
m^{\lambda\lambda'}(\chi'_a)&=\chi_a,&
m^{\lambda\lambda'}(\chi'_b)&=\chi_b+\chi_e,\\
m^{\lambda\lambda'}(\chi'_c)&=\chi_c,&
m^{\lambda\lambda'}(\chi'_d)&=\chi_d+\chi_e.
\end{aligned}
\end{equation}
The defining relations determine $m^{\lambda\lambda'}(\chi'_{e_1})=\chi_b+\chi_c$ and
$m^{\lambda\lambda'}(\chi'_{e_2})=\chi_a+\chi_d$. Let
\begin{equation}
\label{eq-flip-scalar}
\hat{s}^\theta:L_{\lambda'}\to\cx^\times,\qquad
\hat{s}^\theta(k)=\hat{s}(m^{\lambda\lambda'}(k))\theta^{\omega(m^{\lambda\lambda'}(k),\chi_e)}.
\end{equation}
The $n$-th power of this formula is compatible with $\Theta^{\lambda\lambda'}_1$, so
$\hat{s}^\theta$ is an extension of $\sigma'$. Then
$m^{\lambda\lambda'}_{\hat{s}\hat{s}^\theta}:
\rdtorus^\Kas_{q,\sigma'}(\lambda')\to\rdtorus^\Kas_{q,\sigma}(\lambda)$ is the
rescaled monomial map based on $m^{\lambda\lambda'}$ and scalar maps
$\hat{s},\hat{s}^\theta$. (Note $m^{\lambda\lambda'}\circ
m^{\lambda'\lambda}\ne\id$.)

Next we define a special function which is the building block of the representations
of the quantum tori that we consider in this paper, and of the corresponding BB, BLWY
and 1-loop invariants. Recall the cyclic quantum dilogarithm that appeared first
in~\cite{Kashaev:star}
\begin{equation}
\label{Dz}  
D_{q}(x)=\prod_{j=1}^{n-1}(1-q^{j}x)^{j}.
\end{equation}
Let $\Psi_q^\theta:\{x\mid x^n=t_e\}\to\cx$ be a function satisfying the following
properties
\begin{equation}
\label{Psieq}
\Psi_q^\theta(q^2x)=\theta(1+qx)\Psi(x), \qquad \prod_{x^n=t_e}\Psi_q^\theta(x)=1.
\end{equation}
$\Phi_q^\theta$ is uniquely determined (up to multiplication by $\mu_n$) by the
above functional equation. In fact, the functional equation implies that
\begin{equation}
\Psi_q^\theta(x)^n=c D_{q^2}(-qx),
\end{equation}
where $c$ is a normalization constant given below. To avoid consistency issues
with $n$-th roots, fix an $n$-th root $y=t_e^{1/n}$, and write $x=q^{2k}y$. Then
\begin{equation}
\label{Psidef}  
\Psi_q^\theta(q^{2k}y) = c_y \theta^k (-qy;q^2)_k
= c_y \theta^k \prod_{j=1}^k(1+q^{2j-1}y),
\end{equation}
where $c_y$ is another normalization constant such that
\begin{equation}
c_y^n=\theta^{n(n-1)/2}D_{q^2}(-qy),
\end{equation}
which determines the constant
\be
\label{cval}
c=\theta^{n(n-1)/2}=(1+t_e)^{-\frac{n-1}{2}} \,.
\ee

Given $X \in \rdtorus^\Kas_{q,\sigma}(\lambda)$ such that $X^n = t_e$, we can define
$\Psi_q^\theta(X) \in \rdtorus^\Kas_{q,\sigma}(\lambda)$ as $f(X)$ for any
polynomial $f$ such that $f(x)=\Psi_q^\theta(x)$ for $x^n=t_e$. In every
representation, $\Psi_q^\theta(X)$ acts by $\Psi_q^\theta(x)$ in the
$x$-eigenspace.

It is simple to verify that the definitions above give the factorization
\eqref{eq-flip-factor}. The factorization depends on the choices of the extension
$\hat{s}'$ and the root $\theta$. The element $\Psi_q^\theta(X_e)$ has a
$\mu_n$-ambiguity, but the adjoint map does not.

The factorization can be similarly pulled back to the Chekhov--Fock and homology
quantum tori. For the homology part, both maps pull back to the identity. In the
Chekhov--Fock case, the pullback of $m^{\lambda\lambda'}$ is given by
\begin{equation}
\label{eq-lattice-flip}
m^{\lambda\lambda'}(\chi_i)=
\begin{cases}
-\chi_e,&i=e,\\
\chi_i+[\varepsilon_{ie}]_+\chi_e,&i\ne e.
\end{cases}
\end{equation}
Here, $[n]_+=\max(n,0)$ denotes the positive part, and the same adjustment for
self-folded faces applies here. The rest of the factorization goes through
without change.

\begin{remark}
The decomposition when $q$ is generic uses the same monomial map but without
rescaling, and the adjoint is formal with $(-qX_e;q^2)_\infty$ in place of
$\Psi_q^\theta(X_e)$.
\end{remark}

\subsection{Factorization of coordinate changes and intertwiners}
\label{sec-factor-chg}

Given a triangulation $\lambda$ and a mapping class $\varphi$, connect $\lambda$
and $\varphi(\lambda)$ by a sequence of flips
\begin{equation}
\lambda=\lambda_0,\lambda_1,\dotsc,\lambda_{N-1},\lambda_N=\varphi(\lambda)
\end{equation}
where $\lambda_{i+1}$ is the flip of $\lambda_i$ at $e_i\in\lambda_i$. For a
decorated generic character $(r,h)$, we have coordinates
$\sigma_i:L_{\lambda_i}\to\cx^\times$. For convenience, write $\sigma=\sigma_0$.

Choose an extension $\hat{s}_0:L_\lambda\to\cx^\times$ of $\sigma$ and all roots
$\theta_i$ as in the last section, we get the factorization \eqref{eq-flip-factor}
for each flip. Write $\hat{s}_{i+1}=\hat{s}^{\theta_i}$ as in \eqref{eq-flip-scalar}.
Let $m_i=m^{\lambda_{i-1}\lambda_i}_{\hat{s}_{i-1}\hat{s}_i}$ denote the rescaled
monomial map in the factorization of $\Theta^{\lambda_{i-1}\lambda_i}_q$. Then
\begin{equation}
\label{eq-co-chg-decomp}
\begin{split}
  \Theta^{\lambda\varphi(\lambda)}_q&
  =\Theta^{\lambda_0\lambda_1}_q\circ\dotsb\circ\Theta^{\lambda_{N-1},\lambda_N}_q\\
  &=\Ad(\Psi_q^{\theta_0}(X_{e_0}))\circ m_1\circ\dotsb\circ
  \Ad(\Psi_q^{\theta_{N-1}}(X_{e_{N-1}}))\circ m_N\\
&=\Ad\left(\Psi_q^{\theta_0}(X^\#_0)\dotsm\Psi_q^{\theta_{N-1}}(X^\#_{N-1})\right)\circ M_N,
\end{split}
\end{equation}
where
\begin{equation}
  M_i=m_1\circ\dotsb\circ m_i:\rdtorus^\Kas_{q,\sigma_i}(\lambda_i)
  \to\rdtorus^\Kas_{q,\sigma}(\lambda),\qquad
X^\#_i=M_i(X_{e_i})\in\rdtorus^\CFr_{q,\sigma}(\lambda).
\end{equation}
Note each $M_i$ is also a rescaled monomial map. In particular, $X_i^\#$ is a
scaled monomial in $\rdtorus^\CFr_{q,\sigma}(\lambda)$.

Choose an irreducible representation
$\rho_\lambda:\rdtorus^\Kas_{q,\sigma}(\lambda)\to\End(W)$. Then
$\rho'=\rho_\lambda\circ M_N\circ\varphi_\ast$ is a representation with the same
central character as $\rho$ when $(r,h)$ is $\varphi$-invariant. Let $B\in\End(W)$ be
the intertwiner between them, that is, $\rho_\lambda(x)=B^{-1}\cdot\rho'(x)\cdot B$.
Then we have
\begin{equation}
\begin{split}
\rho_\lambda(\Theta^{\lambda\varphi(\lambda)}(\varphi_\ast x))&=
\rho_\lambda\left(\Ad(\Psi_q^{\theta_0}(X^\#_0)\dotsm\Psi_q^{\theta_{N-1}}(X^\#_{N-1}))
  \circ M_N(\varphi_\ast x)\right)\\
&=\Ad\left(\rho_\lambda(\Psi_q^{\theta_0}(X^\#_0)
  \dotsm\Psi_q^{\theta_{N-1}}(X^\#_{N-1}))\right)
\rho_\lambda(M_N(\varphi_\ast x))\\
&=\Ad\left(\rho_\lambda(\Psi_q^{\theta_0}(X^\#_0)
  \dotsm\Psi_q^{\theta_{N-1}}(X^\#_{N-1}))B\right)
\rho_\lambda(x).
\end{split}
\end{equation}
In other words, the intertwiner we want in \eqref{eq-inter-def} for the Kashaev
algebras is given by
\begin{equation}
\label{AB}  
  A^\Kas=\rho_\lambda\left(\Psi_q^{\theta_0}(X^\#_0)
    \dotsm\Psi_q^{\theta_{N-1}}(X^\#_{N-1})\right)B.
\end{equation}

\begin{remark}
\label{rem-theta-choice}
In \cite{Ishi}, Ishibashi imposes the condition that
$\hat{s}_N\circ\varphi_\ast=\hat{s}_0$. In \cite{BB:fiber}, the condition is
satisfied by construction. In our case, the two maps differ by some root of unity in
$\mu_n$.

Both Ishibashi and BB construct intertwiners between a $\varphi$-equivariant sequence
of representations of $\rdtorus^\Kas_{q,\sigma_i}(\lambda_i)$ (or rather the pullback
to the Chekhov--Fock algebra). Since the maps $\hat{s}_i$ are also used in the
representation, the extra condition is necessary. We have eliminated the intermediate
representations in the proof, although it is very useful for numerical calculations.
See Section~\ref{sec-geom}.
\end{remark}

\subsection{Proof of Lemma~\ref{lem-key}}
\label{sub.keylem}

We are now ready to give the proof of the key Lemma~\ref{lem-key}. The statement is
trivial if $p=1$ since there is only one puncture weight. Therefore, we assume
$p\ge2$.

We use the notation of the previous section.
Since $P_\lambda$ is a direct summand of $L_\lambda$ by Lemma~\ref{lem-sub-Kas}, we
can choose a symplectic basis $\{\alpha_i,\beta_i\}$ of $L_\lambda$ that includes a
basis of $P_\lambda$ as the last $p-1$ of $\alpha_i$. We can use this basis to
define $\rho_\lambda$ by \eqref{eq-qtorus-rep} with representation space
$W=V^{\otimes(4g-4+2p)}$. In this construction, the action of $\cx[P_\lambda]$ is
diagonal. Therefore, each component is of the form
\begin{equation}
\label{wzv}  
W_{\zeta_v}=V^{\otimes(4g-3+p)}\otimes(V_{\ell_1}\otimes\dotsm\otimes V_{\ell_{p-1}})
\end{equation}
where $\zeta_i=q^{\ell_i}$.

First we deal with $\rho_\lambda(\Psi_q^{\theta_i}(X_i^\#))$. Since
$X_i^\#\in\rdtorus^\CFr_{q,\sigma} (\lambda)$ is a scaled monomial but not in the
center $\cx[P_\lambda]$, by Lemma~\ref{lem-mon-eig}, the eigenvalues of
$\rho_\lambda(X_i^\#)|_{W_{\zeta_v}}$ are all the $n$-th roots of $t_{e_i}$ with
equal multiplicity $n^{4g-4+p}$. Therefore,
\begin{equation}
\det\left(\rho_\lambda(\Psi_q^{\theta_i}(X_i^\#))|_{W_{\zeta_v}}\right)
=\Big(\prod_{x^n=t_{e_i}}\Psi_q^{\theta_i}(x)\Big)^{n^{4g-4+p}}=1.
\end{equation}

The existence of a basis $W_{\zeta_v}$ such that each $\abs{\det(A_{\zeta_v})}=1$ now
follows from the unitarity statement in Lemma~\ref{lem-unitary}. The
more difficult part is to show that each $\det(A_{\zeta_v})$ is a root of unity.

Continuing the proof, we next deal with $B$ from Equation~\eqref{AB}. 
Recall it is the intertwiner associated to the rescaled monomial map
\begin{equation}
M_N\circ\varphi_\ast=f_{\hat{s}_0,\hat{s}_N\circ\varphi_\ast},\qquad
f=m^{\lambda_0\lambda_1}\circ\dotsb\circ m^{\lambda_{N-1}\lambda_N}\circ\varphi_\ast
\in\Sp(L_\lambda).
\end{equation}
Denote the span of the last $p-1$ $\beta_i$ by $\check{P}_\lambda$. Let
$f_1\in\Sp(L_\lambda)$ that fixes $\alpha_i,\beta_i$ for $i=1,\dotsc,4g-3+p$, acts
like $\varphi_\ast$ on $P_\lambda$, and preserves $\check{P}_\lambda$. This
determines the action on $\check{P}_\lambda$ by the symplectic property. Define
$\hat{s}'':L_\lambda\to\cx^\times$ to be the same as $\hat{s}_N\circ\varphi_\ast$ on
$P_\lambda$ and the same as $\hat{s}_0$ on the other symplectic basis elements. With
$f_2=f_1^{-1}\circ f$, we have
\begin{equation}
  M_N\circ\varphi_\ast=(f_1)_{\hat{s}_0\hat{s}''}\circ (f_2)_{\hat{s}''\hat{s}''}\circ
  \id_{\hat{s}'',\hat{s}_N\circ\varphi_\ast}
\end{equation}
with each map an automorphism of $\rdtorus^\Kas_{q,\sigma}(\lambda)$. Thus,
$B=B_1B_2B_3$ where $B_1$ is an intertwiner between $\rho_\lambda$ and
$\rho_1=\rho_\lambda\circ(f_1)_{\hat{s}_0\hat{s}''}$, $B_2$ is an intertwiner between
$\rho_1$ and $\rho_2=\rho_1\circ(f_2)_{\hat{s}''\hat{s}''}$, and $B_3$ is an
intertwiner between $\rho_2$ and
$\rho'=\rho_2\circ\id_{\hat{s}'',\hat{s}_N\circ\varphi_\ast}$. Note that each $B_i$
also has a block decomposition with respect to $W_{\zeta_v}$. Moreover, $B_2$ and
$B_3$ are block diagonal since they preserve puncture weights.

First, we show that the blocks of $B_1$ are the identity matrix. By construction,
$\rho_\lambda(x^{\alpha_i})$ and $\rho_1(x^{\alpha_i})$ are diagonal for the last
$p-1$ $\alpha_i$. Let $P$ be the permutation on $V^{\otimes(p-1)}$ that realizes the
permutation of their simultaneous eigenvalues. Then it is easy to check that
$B_1=\id\otimes P$. In particular, the blocks $B_1|_{W_{\zeta_v}}$ is the identity
matrix in the standard bases.

Now consider the blocks of $B_2$. Recall that the symplectic group is generated by
transvections $T_a(b)=b+\omega(a,b)a$. $f_2\in\Sp(L_\lambda)$ is in the subgroup
where the restriction to $P_\lambda$ is identity. It is easy to see that this
subgroup is generated by transvections $T_a$ with $a\in P_\lambda^\perp$. The
automorphism $(T_a)_{\hat{s}''\hat{s}''}$ is equal to the conjugation by
\begin{equation}
  E_a=\sum_{k\in\ints/n\ints}q^{k^2}(\hat{s}''(a)^{-1}x^a)^k
  \in\rdtorus_{q,\sigma}(P_\lambda^\perp).
\end{equation}
See also \cite[Eqn.~(8.1)]{Ishi} for a different formulation.
Thus, the intertwiner is the image of this element in the corresponding
representation. Since $x^a$ has the same eigenvalues with the same multiplicities
on each $W_{\zeta_v}$, the intertwiner also has this property. In particular, the
blocks of the intertwiner have the same determinant.

Finally, for $B_3$, the automorphism $\id_{\hat{s}'',\hat{s}_N\circ\varphi_\ast}$ is
purely monomial rescaling. We can rescale one generator at a time. If the scalar maps
only differ on $\alpha_i$, then the intertwiner is $1\otimes\dotsm\otimes
C_i\otimes\dotsm 1$ with $C_i$ a cyclic permutation matrix. If the scalar maps only
differ on $\beta_i$, then the intertwiner is of the same form but with
$C_i=\diag(\xi^j)$ where $\xi\in\mu_n$ is the scalar difference. In both cases, $\det
C_i=1$. Combining the intertwiners from all generators, we see $B_3$ decomposes as
$C_1\otimes\dotsm\otimes C_{4g-4+2p}$ with each $\det C_i=1$. By construction, there
is no rescaling for the last $p-1$ $\alpha_i$, which means that the last $p-1$ $C_i$
are diagonal. This shows that each block $B_3|_{W_{\zeta_v}}$ has a similar form of
tensor decomposition, which implies that each block has determinant $1$.

To summarize, we have constructed an intertwiner $A^\Kas$ such that the determinant
of each block is the same. This implies Lemma~\ref{lem-key} by an overall scaling.
\qed


\section{Examples and computations}
\label{sec.ge}

In this section we finally discuss examples and computations. On the one hand,
examples of mapping classes are efficiently described by
\texttt{flipper}~\cite{flipper}. On the other hand, the corresponding mapping tori
have ideal triangulations whose gluing equations describe $\PSL_2(\BC)$-representations
of the triangulated 3-manifold using the methods of \texttt{SnapPy}~\cite{snappy}. 

\subsection{Gluing equations and Neumann--Zagier data}
\label{sub.NZ}

In this section we briefly review the Neumann--Zagier data associated to
an ideally triangulated 3-manifold. A detailed discussion is given
in~\cite{GZ:neumann} (see also~\cite{GY} for a description close to the examples
of our paper).

Let $\triang$ be an ideal triangulation of a cusped hyperbolic 3-manifold $M$ with
$N$ tetrahedra. Label the vertices of each tetrahedron by 0,1,2,3 such that 1,2,3
appear counterclockwise when viewed from $0$. Assign a variable $z_i$ to the 01 edge
of the $i$-th tetrahedron, and let $z'_i=(1-z_i)^{-1}$, $z''_i=1-z_i^{-1}$.

Choose a peripheral curve on each cusp. The Neumann--Zagier data associated to the
triangulation $\triang$ and these peripheral curves is a tuple $(A,B,\nu)$ where
$A,B$ are $N\times N$ integral matrices, and $\nu\in\ints^N$. They define the gluing
equations
\begin{equation}
z^Az^{\prime\prime B}=(-1)^\nu
\end{equation}
whose solutions (with some extra conditions) can be used to describe the hyperbolic
structure. Each $z^\square_i$ is a cross-ratio of the vertices of the tetrahedron
lifted to $\AMSbb{H}^3$.

An important property of the Neumann--Zagier data is that the matrix $(A|B)$ has
full rank over $\ints$. This implies that given a solution to the gluing
equations, we can find logarithm branches for all $z,z''$ such that the
logarithmic equation $A\log z+B\log z''=\pi\iunit \nu$ has a solution.

\subsection{Layered triangulations}

Let $M_\varphi$ be the mapping torus of the mapping class $\varphi$. A triangulation
$\triang$ of $M_\varphi$ with tetrahedra $T_1,\dotsc,T_N$ is layered if there are
cellular immersions $\lambda_i\to\triang^{(2)}, i\in\ints$ from triangulations
$\lambda_i$ of $\surface$ to the 2-skeleton $\triang^{(2)}$ of $\triang$ such
that $\lambda_{i+N}=\varphi(\lambda_i)$, 
$\lambda_0,\dotsc,\lambda_N$ is a sequence of flips such that the immersions of
$\lambda_i$ and $\lambda_{i+1}$ are related by the layered tetrahedron
$T_{i\bmod{N}}$, which is labeled like Figure~\ref{fig-layered-T}. Here, the bottom
two faces belong to $\lambda_i$, and the top two belong to $\lambda_{i+1}$.

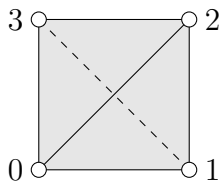
\begin{figure}
\centering
\begin{tikzpicture}
\tikzmath{\r=2;}
\draw[fill=gray!20] (0,0) rectangle (\r,\r);
\draw (0,0) -- (\r,\r);
\draw[dashed] (0,\r) -- (\r,0);
\draw[fill=white,radius=0.1,inner sep=0.2cm] (0,0)circle node[left]{$0$} (\r,0)
circle node[right]{$1$} (\r,\r)circle node[right]{$2$} (0,\r)circle node[left]{$3$};
\end{tikzpicture}
\caption{Labels of a layered tetrahedron.}
\label{fig-layered-T}
\end{figure}

It is clear that the layered triangulation $\triang$ has a canonical taut angle
structure by assigning $\pi$ to the edges 02 and 13 in each tetrahedron. In fact
more is true.
Each face in the layered triangulation $\triang$ inherits an orientation from the
immersions of $\lambda_i$. Given $e\in\lambda_i$, let $F_1,F_2$ be the images in
$\triang^{(2)}$ of the two faces of $\lambda_i$ adjacent to $e$. The set of edges of
tetrahedra that are identified to $e$ is partition by $F_1,F_2$ into those above and
below $\lambda_i$. See Figure~\ref{fig-around-edge}.

\begin{figure}
\centering
\begin{tikzpicture}
\tikzmath{\r0=2;\c=0.3;\h0=1;\r1=\r0-\c;\h1=\h0+\c;}
\draw (-\r1,-\h1) to[out=20,in=180] (-\c,-\h0)coordinate(bottom)
  to[out=0,in=160] (\r1,-\h1)
  to[out=105,in=-45] (\c,\h0)coordinate(top)
  to[out=-150,in=75] cycle;
  \draw[fill=gray!20] (-\r0,0)node[left]{$F_1$} -- (bottom)
  -- (\r0,0)node[right]{$F_2$} -- (top) -- cycle;
\draw[fill=white] (-\r1,\h1) to[out=-75,in=135] (bottom)
  to[out=30,in=-105] (\r1,\h1)
  to[out=-160,in=0] (top)
  to[out=180,in=-20] cycle;
\draw (top) --node[left]{$e$} (bottom);
\end{tikzpicture}
\caption{Tetrahedra around the edge $e\in\lambda_i$}
\label{fig-around-edge}
\end{figure}
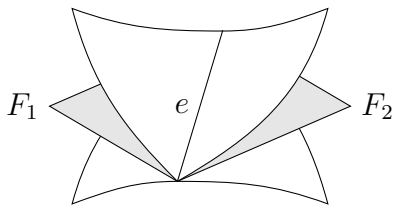

\subsection{The geometric representation}
\label{sec-geom}

Suppose $\varphi$ is pseudo-Anosov. Then $M_\varphi$ has a complete hyperbolic
structure, whose holonomy $\pi_1(M_\varphi)\to\PSL_2(\cx)$ is discrete and faithful.
As a result, the restriction $r:\pi_1(\surface)\to\PSL_2(\cx)$ is a
$\varphi$-invariant generic character of the surface $\surface$.

Let $\triang$ be the layered triangulation of $M_\varphi$. By the discussion at the
end of \cite[Section~3]{DGY}, $\triang$ admits Ptolemy assignments, so $r$ admits
decorations. Moreover, the gluing equations of $\triang$ has a solution corresponding
to the hyperbolic structure, and the shear-bend coordinates $t_e=\sigma_i(\chi_e)$ of
all $e\in\lambda_i$ can be calculated from this solution. From
Figure~\ref{fig-around-edge}, we see the shear-bend coordinate of $e\in\lambda_i$ is
minus the product of the shape parameters above $e$. As a simple example, the
shear-bend coordinates of the flipped edges in $\lambda_i$ and $\lambda_{i+1}$ are
$-z'_i$ and $-1/z'_i$ respectively using the convention above. Conversely, given
shear-bend coordinates of $\lambda_0$, the coordinates of all $\lambda_i$ are
determined using the coordinate change $\Theta^{\lambda_0\lambda_i}_1$. Then $z'_i$
are determined by shear-bend coordinates as above. Note that this implies
$\theta_i^n=z''_i$ where $\theta_i$ is from Section~\ref{sec-factor-chg}.

Since we can always find logarithmic branches given a solution to the Neumann--Zagier
equations, the shear-bend coordinates also acquire logarithmic branches. Our choice
is that the logarithm of a shear-bend coordinate is the corresponding sum of the
logarithmic shape parameters minus $\pi\iunit$. Continuing the example of the flipped
edges, the logarithmic shear-bend are $\log(z'_i)-\pi\iunit$ and
$\pi\iunit-\log(z'_i)$ respectively.

The existence of consistent logarithmic branches implies that the scalar maps
$\hat{s}_i$ and roots $\theta_i$ used in Section~\ref{sec-factor-chg} can be chosen
so that $\hat{s}^N\circ\varphi_\ast=\hat{s}_0$. This is not necessary for the
construction, but it simplifies calculation.

\subsection{Explicit calculations}

The invariants considered in this paper are suitable for calculations. This has
been implemented, and the code is available at
\url{https://github.com/newcworld001/BBBLWY}.

Flip descriptions of mapping classes are available in \texttt{flipper}, and
\texttt{SnapPy} converts \texttt{flipper} mapping classes into the layered
triangulation of the mapping torus with a solution to the gluing equations. Then
we can calculate all shear-bend coordinates as well as their logarithmic
branches as explained above.

The quantum torus is implemented using a modified class from SageMath
\cite{sagemath}, and representations are implemented using
\eqref{eq-qtorus-rep}. From this, the $\gl_1$-invariant
$A^H_{\varphi,\zeta_v}(q)$ can be calculated by Lemma~\ref{lem-unitary}. The
BLWY intertwiner $A^\CF_{\varphi,r,\zeta_v}(q)$ is similar.

For the BB intertwiner $A^\Kas_{\varphi,r}(q)$, we use Ishibashi's explicit
formulas in \cite{Ishi} since they only use the shear-bend coordinates.
Alternatively, one can lift the shear-bend coordinates to Kashaev coordinates
and use the same algorithm as the other two invariants.

In the rest of the section we discuss examples of genus 1 surfaces with 1 and 2
punctures. 

\subsection{Examples on $\surface_{1,1}$}
\label{sub.ex11}

The standard and simplest example in genus 1 with 1 puncture
is the complement of the $4_1$ knot, which has fiber $\surface_{1,1}$
and monodromy $T_aT_b^{-1}$, where $a,b$ are the standard curves, and $T$ denotes
Dehn twists. The BLWY invariant is given explicitly in \cite{BWY:I,GY} with
slightly different conventions. The evaluation of these formulas at
$q=\exp(\frac{16}{15}\pi\iunit)$ agrees with the numerical values of BB invariants
from \cite[Table~3]{BB:analytic}. This is consistent with \eqref{inv-dec3}.

\subsection{Examples on $\surface_{1,2}$}
\label{sub.ex12}

For more interesting examples, consider the surface $\surface_{1,2}$. Let
$a,b,c$ be a chain of curves, and $x$ be the ideal arc that intersects $c$ at a
point and is disjoint from the others. See Figure~\ref{fig-S12}, where a
triangulated picture is also given. As usual, opposite sides of the square are
identified. Let $T_a,T_b,T_c$ denote the Dehn twists, and $X$ denote the half
twist along $x$.

\begin{figure}
\centering
\begin{tikzpicture}[baseline=(ref.base)]
\draw (0,0)ellipse[x radius=2,y radius=1];
\draw (-0.5,0) to[curve through={(0,0.1)}] (0.5,0);
\draw (-0.8,0.2) to[curve through={(-0.5,0)..(0,-0.1)..(0.5,0)}] (0.8,0.2);
\begin{scope}[thick]
\draw (0,0)ellipse[x radius=1.2,y radius=0.6];
\path (1.2,0)node[right]{$b$};
\draw (0,-1)arc[start angle=-90,end angle=90,x radius=0.2,y radius=0.45];
\draw[dashed] (0,-0.1)arc[start angle=90,end angle=270,
	x radius=0.2,y radius=0.45];
\path (0,-1)node[below]{$c$};
\draw (0,0.1)arc[start angle=-90,end angle=90,x radius=0.2,y radius=0.45];
\draw[dashed] (0,1)arc[start angle=90,end angle=270,
	x radius=0.2,y radius=0.45];
\path (0,1)node[above]{$a$};
\draw (-2,0)arc[x radius=2,y radius=1,start angle=-180,end angle=0]
  node[near end,below right]{$x$};
\end{scope}
\draw[fill=white] (-2,0)circle[radius=0.1] (2,0)circle[radius=0.1];
\tikzbase{0,0};
\end{tikzpicture}
\qquad
\begin{tikzpicture}[baseline=(ref.base)]
\draw (-1,-1) rectangle (1,1) (-1,-1) -- (1,1) (-1,1) -- (1,-1);
\begin{scope}[thick]
\draw (-0.6,-1) -- ++(0,2) node[above]{$a$};
\draw (0.4,-1) -- ++(0,2) node[above]{$c$};
\draw (-1,-0.5) -- node[below,inner sep=1pt]{$b$} ++(2,0);
\draw (0,0) -- node[pos=0.6,below right,inner sep=1pt]{$x$} (1,1);
\end{scope}
\draw[fill=white,radius=0.1] (0,0)circle[]
  foreach \x in {-1,1} {foreach \y in {-1,1} {(\x,\y)circle[]}};
\tikzbase{0,0};
\end{tikzpicture}
\caption{Curves on $\surface_{1,2}$}\label{fig-S12}
\end{figure}
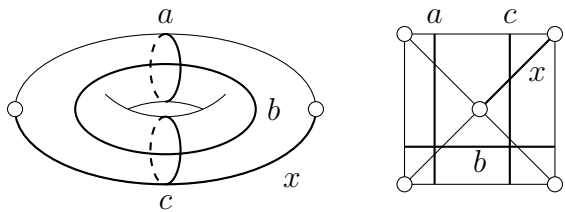

The working precision of the following examples is 256 bits or roughly 77
decimal digits. All numerical equalities below hold to at least 75 decimal
places.

\begin{enumcc}[$\bullet$]
\item $\varphi=T_b^5T_aT_c^{-1}$.
Its mapping torus is \texttt{t09265} in the \texttt{SnapPy} census. $\varphi$
can be realized as a sequence of 11 flips, part of which is shown in
Figure~\ref{fig-example-flips}. By the convention of \texttt{flipper}, the new
edge after flip has the same label as before. Each edge also has a choice of
positive direction used by \texttt{flipper} for many purposes. Here, it is
useful because it breaks the symmetry of the surface, so the final triangulation
can be uniquely identified with the initial one by a relabeling. In this
example, the final edge 0 is identified with the reversed initial edge 0.

\begin{figure}
\centering
\begin{tikzpicture}[baseline=(ref.base)]
\draw (-1,-1) rectangle (1,1) (-1,-1) -- (1,1) (-1,1) -- (1,-1);
\begin{scope}[every node/.style={circle,fill=white,inner sep=1pt,node font=\scriptsize}]
\path (-1,0)node{$0$} (1,0)node{$0$} (0,-1)node{$3$} (0,1)node{$3$};
\path (0.5,0.5)node{$1$} (-0.5,0.5)node{$2$}
  (-0.5,-0.5)node{$4$} (0.5,-0.5)node{$5$};
\end{scope}
\tikzmath{\r=0.85;}
\path[-o-={\r}{>}] (0,0) -- (1,1); 
\path[-o-={\r}{>}] (-1,-1) -- (0,0); 
\path[-o-={\r}{>}] (-1,1) -- (0,0); 
\path[-o-={\r}{>}] (0,0) -- (1,-1); 
\path[-o-={\r}{>}] (-1,1) -- (-1,-1); 
\path[-o-={\r}{>}] (1,1) -- (1,-1); 
\path[-o-={\r}{>}] (1,1) -- (-1,1); 
\path[-o-={\r}{>}] (1,-1) -- (-1,-1); 
\draw[fill=white,radius=0.1] (0,0)circle[]
  foreach \x in {-1,1} {foreach \y in {-1,1} {(\x,\y)circle[]}};
\tikzbase{0,0};
\end{tikzpicture}
$=$
\begin{tikzpicture}[baseline=(ref.base)]
\draw (-1,-1) -- (1,1) -- (0,2) -- (-1,1) -- (1,-1) -- (1,1) -| cycle;
\begin{scope}[every node/.style={circle,fill=white,inner sep=1pt,node font=\scriptsize}]
\path (-1,0)node{$0$} (1,0)node{$0$}
  (-0.5,-0.5)node{$4$} (-0.5,1.5)node{$4$}
  (0.5,-0.5)node{$5$} (0.5,1.5)node{$5$};
\path (0.5,0.5)node{$1$} (-0.5,0.5)node{$2$} (0,1)node{$3$};
\end{scope}
\draw[fill=white,radius=0.1] (0,0)circle[] (0,2)circle[]
  foreach \x in {-1,1} {foreach \y in {-1,1} {(\x,\y)circle[]}};
\tikzbase{0,0};
\end{tikzpicture}
$\xrightarrow{\text{Flip }3}$
\begin{tikzpicture}[baseline=(ref.base)]
\draw (-1,1) -- (-1,-1) -- (1,1) -- (0,2) -- (-1,1) -- (1,-1) -- (1,1)
  (0,0) -- (0,2);
\begin{scope}[every node/.style={circle,fill=white,inner sep=1pt,node font=\scriptsize}]
\path (-1,0)node{$0$} (1,0)node{$0$}
  (-0.5,-0.5)node{$4$} (-0.5,1.5)node{$4$}
  (0.5,-0.5)node{$5$} (0.5,1.5)node{$5$};
\path (0.5,0.5)node{$1$} (-0.5,0.5)node{$2$} (0,1)node{$3$};
\end{scope}
\draw[fill=white,radius=0.1] (0,0)circle[] (0,2)circle[]
  foreach \x in {-1,1} {foreach \y in {-1,1} {(\x,\y)circle[]}};
\tikzbase{0,0};
\end{tikzpicture}
$\xrightarrow{\text{Flips }1,2}$
\begin{tikzpicture}[baseline=(ref.base)]
  \draw (-1,-1) -- (0,0) -- (1,-1) -- (1,1) -- (0,2) -- (-1,1) -- cycle -- (0,2)
  -- (1,-1) (0,0) -- (0,2);
\begin{scope}[every node/.style={circle,fill=white,inner sep=1pt,node font=\scriptsize}]
\path (-1,0)node{$0$} (1,0)node{$0$}
  (-0.5,-0.5)node{$4$} (-0.5,1.5)node{$4$}
  (0.5,-0.5)node{$5$} (0.5,1.5)node{$5$};
\path (0.5,0.5)node{$1$} (-0.5,0.5)node{$2$} (0,1)node{$3$};
\end{scope}
\draw[fill=white,radius=0.1] (0,0)circle[] (0,2)circle[]
  foreach \x in {-1,1} {foreach \y in {-1,1} {(\x,\y)circle[]}};
\tikzbase{0,0};
\end{tikzpicture}
\vskip1em
$=$
\begin{tikzpicture}[baseline=(ref.base)]
\draw (-1,1) -- (-2,2)  -- (-1,-1) -- (0,0) -- (1,-1) -- (0,2) -- cycle
  -- (-1,-1) -- (0,2) -- (0,0);
\begin{scope}[every node/.style={circle,fill=white,inner sep=1pt,node font=\scriptsize}]
\path (-1.5,0.5)node{$1$} (0.5,0.5)node{$1$} 
  (-0.5,-0.5)node{$4$} (-0.5,1.5)node{$4$}
  (0.5,-0.5)node{$5$} (-1.5,1.5)node{$5$};
\path (-1,0)node{$0$} (-0.5,0.5)node{$2$} (0,1)node{$3$};
\end{scope}
\draw[fill=white,radius=0.1] foreach \p in {(0,0),(0,2),(1,-1),(-1,1),(-1,-1),(-2,2)}
{\p circle[]};
\tikzbase{0,0};
\end{tikzpicture}
$\xrightarrow{\text{Flips }0,3}$
\begin{tikzpicture}[baseline=(ref.base)]
\draw (-1,1) -- (-2,2)  -- (-1,-1) -- (0,0) -- (1,-1) -- (0,2) -- cycle
  (-1,-1) -- (0,2);
\draw (-2,2) ..controls(-1,-0.2).. coordinate(M0) (0,2)
  (-1,-1) ..controls(0,1.2).. coordinate(M3) (1,-1);
\begin{scope}[every node/.style={circle,fill=white,inner sep=1pt,node font=\scriptsize}]
\path (-1.5,0.5)node{$1$} (0.5,0.5)node{$1$} 
  (-0.5,-0.5)node{$4$} (-0.5,1.5)node{$4$}
  (0.5,-0.5)node{$5$} (-1.5,1.5)node{$5$};
\path (M0)node{$0$} (-0.5,0.5)node{$2$} (M3)node{$3$};
\end{scope}
\draw[fill=white,radius=0.1] foreach \p in {(0,0),(0,2),(1,-1),(-1,1),(-1,-1),(-2,2)}
{\p circle[]};
\tikzbase{0,0};
\end{tikzpicture}
$\xrightarrow{\text{Flips }2,1,2,1,2,0}$
\begin{tikzpicture}[baseline=(ref.base)]
\draw (-1,-1) rectangle (1,1) (-1,-1) -- (1,1) (-1,1) -- (1,-1);
\begin{scope}[every node/.style={circle,fill=white,inner sep=1pt,node font=\scriptsize}]
\path (-1,0)node{$0$} (1,0)node{$0$} (0,-1)node{$3$} (0,1)node{$3$};
\path (0.5,0.5)node{$5$} (-0.5,0.5)node{$4$}
  (-0.5,-0.5)node{$1$} (0.5,-0.5)node{$2$};
\end{scope}
\tikzmath{\r=0.85;}
\path[-o-={\r}{>}] (0,0) -- (1,1); 
\path[-o-={\r}{>}] (0,0) -- (-1,-1); 
\path[-o-={\r}{>}] (-1,1) -- (0,0); 
\path[-o-={\r}{>}] (1,-1) -- (0,0); 
\path[-o-={\r}{>}] (-1,-1) -- (-1,1); 
\path[-o-={\r}{>}] (1,-1) -- (1,1); 
\path[-o-={\r}{>}] (-1,1) -- (1,1); 
\path[-o-={\r}{>}] (-1,-1) -- (1,-1); 
\draw[fill=white,radius=0.1] (0,0)circle[]
  foreach \x in {-1,1} {foreach \y in {-1,1} {(\x,\y)circle[]}};
\tikzbase{0,0};
\end{tikzpicture}
\caption{Part of the sequence of flips for $T_b^5T_aT_c^{-1}$}\label{fig-example-flips}
\end{figure}
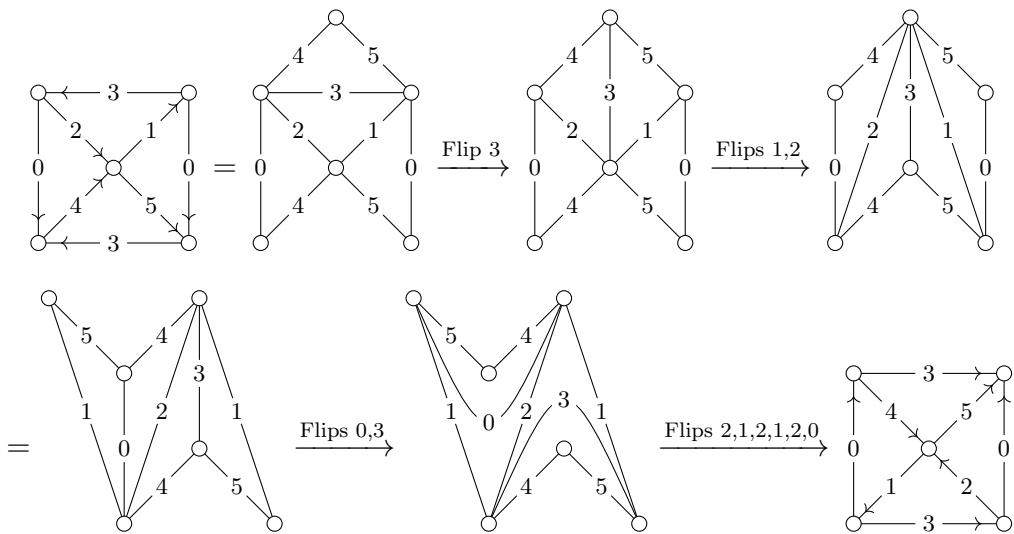

An easy calculation from surgery diagrams give
\begin{equation}
\abs{T^H_{\varphi,\zeta_v}(q)}=
\begin{cases}
d\sqrt{n/d},&\zeta_v=(1,1),\\
0,&\text{otherwise},
\end{cases}
\end{equation}
where $d=\gcd(n,5)$. Then Theorem~\ref{thm.1} says
\begin{equation}
\abs{T^\Kas_{\varphi,r}(q)}=d\sqrt{n/d}\abs{T^\CF_{\varphi,r,1}(q)}.
\end{equation}
Numerical calculations confirm this. For $q=\exp(2\pi\iunit/3)$, we find
\begin{equation}
\abs{T^\Kas_{\varphi,r}(q)}=\sqrt{3}\abs{T^\CF_{\varphi,r,1}(q)}=13.444319.
\end{equation}
On the other hand, for $q=\exp(2\pi\iunit/5)$, we find
\begin{equation}
\abs{T^\Kas_{\varphi,r}(q)}=5\abs{T^\CF_{\varphi,r,1}(q)}=31.451090.
\end{equation}
\item $\varphi=T_b^{-1}T_aT_cX$.
Its mapping torus is \texttt{s254} in
the \texttt{SnapPy} census. If the punctures are filled in, $X$ is trivial, and
$a=c$, so the mapping class becomes $\widehat{\varphi}=T_b^{-1}T_a^2$, whose mapping
torus has homology $\ints/2\oplus\ints$. By Corollary~\ref{cor-TH-values}, all
$\gl_1$-invariants have absolute value 1. Since $X$ permutes the punctures, the only
invariant puncture weights are $\zeta_v=(1,1)$. By Theorem~\ref{thm.1},
$\abs{T^\Kas_{\varphi,r}(q)}=\abs{T^\CF_{\varphi,r,1}(q)}$. We confirm this
numerically at $q=\exp(2\pi\iunit/3)$ which gives
\begin{equation}
\abs{T^\Kas_{\varphi,r}(q)}=\abs{T^\CF_{\varphi,r,1}(q)}=4.19825.
\end{equation}
\item $\varphi=T_a^2T_bT_c^{-1}$.
Its mapping torus is the $9^2_{50}$ link complement. Again, we find that all
$\gl_1$-invariants
have absolute value 1. Since the punctures are not permuted by $\varphi$, now all
puncture weights contribute. At $q=\exp(2\pi\iunit/3)$, we find numerically
\begin{equation}
T^\Kas_{\varphi,r}(q)=-T^\CF_{\varphi,r,(1,1)}(q)-2T^\CF_{\varphi,r,(q,q^{-1})}(q).
\end{equation}
Here, the factor $2$ is due to the symmetry of the surface and the mapping class.
Note the minus signs come from the $\gl_1$-invariants, not the sign in
Theorem~\ref{thm.1}.
\end{enumcc}


\subsection*{Acknowledgements} 

The authors wish to thank St\'ephane Baseilhac, Francis Bonahon, Thang L\^e,
Helen Wong, Tian Yang and Campbell Wheeler for many enlightening conversations. 


\bibliographystyle{hamsalpha}
\bibliography{biblio}
\end{document}